\documentclass[]{article}

\title{\bd{THE SCROLLAR INVARIANTS OF CURVES MAPPING TO A HIRZEBRUCH SURFACE}}
\author{\small RICCARDO REDIGOLO}
\date{}

\usepackage{amsmath,amssymb,amsthm,amsfonts}
\usepackage{braket}
\usepackage[bb=libus]{mathalpha}
\usepackage{extpfeil}

\usepackage[
backend=biber,
style=alphabetic,
sorting=nyt,
maxbibnames=99,
maxcitenames=99
]{biblatex}
\defbibheading{bibliography}[\refname]{%
	\section*{\centering \MakeUppercase{REFERENCES}}%
}

\usepackage[margin=1.35in]{geometry}

\usepackage{graphicx}
\graphicspath{{./images/}}
\usepackage{tikz-cd}
\usepackage{tikz}
\usetikzlibrary{fit, backgrounds}
\usetikzlibrary{decorations.pathreplacing, calc}
\usetikzlibrary{positioning}
\usepackage{ytableau}

\usepackage{xcolor}
\definecolor{commentgreen}{RGB}{2,150,10}
\definecolor{eminence}{RGB}{140,48,170}
\definecolor{weborange}{RGB}{180,90,0}
\definecolor{frenchplum}{RGB}{129,20,83}
\definecolor{betterblue}{RGB}{0,90,150}
\definecolor{greyy}{RGB}{150,150,150}
\definecolor{dkgreen}{rgb}{0,0.6,0}
\definecolor{gray}{rgb}{50,50,50}
\definecolor{mauve}{rgb}{0.58,0,0.82}

\usepackage{listings}
\usepackage{hyperref}

\usepackage{fancyhdr}
\newcommand{\address}[1]{%
  \par\smallskip
  \noindent\textsl{Current address: #1}\par
}

\newcommand{\email}[1]{%
  \noindent\textsl{E-mail address: }\texttt{#1}\par
}

\usepackage{enumitem}

\newcommand{\bd}{\textbf}
\newcommand{\mc}{\mathcal}
\newcommand{\mcb}{\mathcalboondox}
\newcommand{\mbd}{\mathbi}

\newcommand{\ra}{\rightarrow}

\newcommand{\ol}{\overline}

\newcommand{\Pb}{\mathbb{P}}
\newcommand{\Proj}{{\mathbb{P}^1}}

\newcommand{\Tsch}{\mathcal{T}\!sch}

\newtheorem{thm}{Theorem}[section]

\newtheorem{prop}[thm]{Proposition}
\newtheorem{conj}[thm]{Conjecture}
\newtheorem{lem}[thm]{Lemma}
\newtheorem{cor}[thm]{Corollary}

\newtheoremstyle{remarkupright}
{\topsep} {\topsep} {\upshape} {} {\bfseries\upshape} {.} {.5em} {}
\theoremstyle{remarkupright}
\newtheorem{oss}[thm]{Remark}
\newtheorem*{akn}{Acknowledgments}
\newtheorem{ex}{Example}[section]

\DeclareFontFamily{U}{BOONDOX-calo}{\skewchar\font=45}
\DeclareFontShape{U}{BOONDOX-calo}{m}{n}{<-> s*[1.05] BOONDOX-r-calo}{}
\DeclareFontShape{U}{BOONDOX-calo}{b}{n}{<-> s*[1.05] BOONDOX-b-calo}{}
\DeclareMathAlphabet{\mathcalboondox}{U}{BOONDOX-calo}{m}{n}
\SetMathAlphabet{\mathcalboondox}{bold}{U}{BOONDOX-calo}{b}{n}
\DeclareMathAlphabet{\mathbcalboondox}{U}{BOONDOX-calo}{b}{n}
\def\mathbi#1{\textbf{\em #1}}

\makeatletter
\renewcommand{\maketitle}{
	\begin{center}
		{\large\normalfont\bfseries \@title \par}
		\vskip 2em
		{\normalfont \@author \par}
		\vskip 4em
	\end{center}
}

\renewenvironment{abstract}{
	\quotation
	\noindent\small\text{ABSTRACT.}\quad
}{
	\endquotation
}
\makeatother

\usepackage{titlesec}
\titleformat{\section}
{\centering\small}
{\thesection.}
{1em}
{}

\makeatletter
\renewcommand\section{\@startsection {section}{1}{\z@}%
	{-3.5ex \@plus -1ex \@minus -.2ex}%
	{2.3ex \@plus.2ex}%
	{\centering\normalfont\small}}
\renewcommand\thesection{\arabic{section}}
\makeatother

\makeatletter
\renewenvironment{proof}[1][\proofname]{%
	\par
	\pushQED{\qed}%
	\normalfont\normalsize\topsep6\p@\@plus6\p@\relax
	\trivlist
	\item[\hskip\labelsep\itshape #1\@addpunct{.}]
}{%
	\popQED\endtrivlist\@endpefalse
}
\makeatother

\begin{document}
	\maketitle
	\thispagestyle{empty}
	\begin{abstract}
		In this note we analyse the scrollar invariants of $k:1$ covers of $\Proj$ that factor through the normalisation of a nodal curve in the $m$-th Hirzebruch surface $\mathbb F_m$. We then give an existence theorem for nodal curves in $\mathbb F_m$ having fixed class and singular locus. 
	\end{abstract}
	\section{INTRODUCTION}
	\normalfont
	Given a cover $\alpha:C\ra B$ between two smooth irreducible curves, the inclusion ${\alpha^\#:\mcb O_Y\ra \alpha_*\mcb O_X}$ gives rise to the following split exact sequence:
	\[0\ra \mcb O_Y\ra \alpha_*\mcb O_X\ra \Tsch(\alpha)^\vee\ra 0\]
	One calls $\Tsch(\alpha)$ the Tschirnhausen bundle of $\alpha$. This bundle has historically played a central role in the study of covers of curves, with its use to study trigonal curves dating back to the first half of the XX-th century (see for example the classical result of A. Maroni \cite{Maroni} or, for a more modern treatment, \cite{Miranda1985TripleCovers}). For foundational results on Tschirnhausen bundles see \cite{CE}, where, more generally, they study Gorenstein covers of algebraic varieties.
	\vskip 1 em
	When $B=\Pb^1_\mcb k$, the bundle splits as
	\[\Tsch(\alpha)\cong \mcb O_{\Proj}(e_1)\oplus \dots \oplus \mcb O_{\Proj}(e_{k-1}),\]
	with $0<e_1\leq \dots \leq e_{k-1}$. These integers are called the scrollar invariants of $\alpha$ and their study has shown to play a crucial role in the analysis of syzygies of algebraic curves, see for example the paper of F. O. Schreyer \cite{schreyer} or, for a more recent example, the one by G. Farkas and M. Kemeny \cite{farkasKemeny}. More broadly, the study of how line bundles on $k$-gonal curves split when pushed forward to $\Pb^1_\mcb k$ has shown to be central in studying the Brill Noether theory of those curves. Indeed, their study has allowed H. Larson in \cite{LHk} to refine an earlier result of N. Pflueger \cite{pfl} and determine the dimension of the irreducible components of the Brill-Noether loci for general $k$-gonal curves. This has then led to many interesting developments in Hurwitz-Brill-Noether theory like \cite{LLV}, \cite{CookPowellJensen2022} or \cite{farkasFeyzbakhshRojas}.
	\vskip 1 em
	Calling $\mc H_{g,k}$ the Hurwitz stack parametrising degree $k$ covers of $\Pb^1_\mcb k$ by a genus $g$ curve, it is a well-known result of \cite{ball89} that the scrollar invariants of a general cover $\alpha\in\mc H_{g,k}$, are balanced (i.e. $e_{k-1}-e_1\leq 1$). See also the work of I. Coskun, E. Larson and I. Vogt \cite{CoskunLarsonVogt2024} for the case where $\alpha$ is a primitive cover between curves of positive genus. As it often happens though, the general behaviour is very different from the one exhibited by the covers we are most familiar with. For example, if one takes $C$ to be a smooth degree $k$ curve in $\Pb^2_\mcb k$ and $\alpha:C\ra \Pb^1_\mcb k$ to be the projection from any point not in $C$, one easily checks that:
	\[
	\Tsch(\alpha)\cong \mcb O_{\Proj}(1)\oplus \dots \oplus \mcb O_{\Proj}(k-1)
	\]
	A natural question to ask is then which tuples of natural numbers can arise as the scrollar invariants of a primitive $k:1$ cover of $\Pb^1_\mcb k$. This problem has recently received the attention of R. Vakil and S. Vemulapalli who, in \cite{VV}, conjecture the following:
	\begin{conj}\label{VV conj} $\mathrm{(Vemulapalli}$, Conjecture $1.3$ in \cite{VV}) 
		Let $k\geq 3$, $g\geq 0$ and consider the following polytope:
		\[\mc P_k:=\left\{\mbd x\in \mathbb R^{k-1}\ |\ \sum_{i\leq k-1}x_i=1,\ 0\leq x_1\leq\dots\leq x_{k-1},\ x_{i+j}\leq x_i+x_j\right\}\]
		A tuple of positive integers $\mbd e:=(e_1,\dots,e_{k-1})$ with $e_1+\dots+e_{k-1}=g+k-1$ arises as the scrollar invariants of a primitive degree $k$ cover $\pi:C\ra \Pb^1_\mcb k$ from a genus $g$ curve if and only if $\left(\frac{ e_1}{g+k-1},\dots,\frac{ e_1}{g+k-1}\right)\in\mc P_d$. 
	\end{conj}
	A useful tool in tackling this problem has been the study of curves that are normalisation of some singular curves lying on a Hirzebruch surface $\mathbb F_m$ (see for example Theorem $1.5$ of \cite{VV}).
	\vskip 1 em
	The aim of this paper is to describe the scrollar invariants of the normalisations of nodal curves in a Hirzebruch surface. Throughout we shall work over an algebraically closed field $\mcb k$ of characteristic $0$ and will denote by $\pi:\mathbb F_m:=\Pb(\mcb O_\Proj\oplus \mcb O_\Proj(m))\ra \Proj$ the $m$-th Hirzebruch surface. For convenience, calling $\zeta:=c_1(\mcb O_{\mathbb F_m}(1))$ and $\xi:=\pi^*c_1(\mcb O_\Proj(1))$, we shall identify $\mathrm{Pic}(\mathbb F_m)$ with $\mathbb Z^{\oplus 2}$ and say that a divisor of class $k\zeta+a\xi$ is of type $(k,a)$.
	\vskip 1 em
	Let $\mc M_g(\mathbb F_m,(k,a))$ be the stack parametrising maps $f:C\ra \mathbb F_m$ where $C$ is a smooth curve of genus $g$ and $f_*[C]=(k,a)$. We will denote the main component of $\mc M_g(\mathbb F_m,(k,a))$, i.e. the one dominating the Severi variety parametrising irreducible curves in $\mathbb F_m$ of class $(k,a)$ and geometric genus $g$, by $\mc M_g(\mathbb F_m,(k,a))^{\mathrm{bir}}$. Notice that there is a natural map
	\[\mc M_g(\mathbb F_m,(k,a))^{\mathrm{bir}}\ra \mc H_{g,k}\]
	given by post composing with $\pi$. It is then natural to ask what are the scrollar invariants of a general point in the image of this map. Since, as noted for example in \cite{christ2023severi}, $\mc M_g(\mathbb F_m,(k,a))^{\mathrm{bir}}$ is non-empty if and only if $g\leq p_a(k,a):=\binom{k}{2}m+(k-1)(a-1)$, in what follows, we will always assume this condition to hold.
	\begin{thm}\label{thm 1}
		Let $\nu:C\ra \Gamma\subseteq \mathbb F_m$ be a general point of $\mc M_g(\mathbb F_m,(k,a))^{bir}$. Set $\alpha:=\pi\circ \nu$ and ${\delta:=p_a(\Gamma)-g}$. Assume that $(m,k,a)\notin \{(2,4,0),\ (1,6,0),\ (1,4,2),\ (0,4,4)\}$ and that 
		\[3\delta\leq |\mcb O_{\mathbb F_m}(k,a)|.\]
		Then, the Tschirnhausen bundle of $\alpha$ is balanced if and only if $\delta\geq \binom{k-1}{2}m$.\\
		If $\delta< \binom{k-1}{2}m$, choosing $\ell$ and $d$ such that $\binom{\ell+1}{2}m>\delta = \binom{\ell}{2}m+d$, the scrollar invariants of $C$ are:
		\[e_i=\begin{cases}
			im+a & \text{if } i\leq k-1-\ell\\
			(k-\ell)m+a-\lceil\frac{d}{\ell}\rceil                           & \text{if $k-\ell\leq i< k-\ell+j$}\\
			(k-\ell)m+a-\lfloor\frac{d}{\ell}\rfloor                                 & \text{if } i\geq k-\ell+j
		\end{cases}\]
		where $0\leq j<\ell$ is equal to $d$ $mod$ $\ell$.
	\end{thm}
	Note that Theorem \ref{thm 1} shows that also in the second case the Tschirnhausen bundle of the cover is balanced to some extent.
	\vskip 1 em
	Building on this, the next natural question to ask should be what is the dimension of the Brill-Noether loci of these curves. When $g=p_a(k,a)$, a general point of $\mc M_g(\mathbb F_m,(k,a))^{bir}$ is an embedding and this problem has recently been solved by H. Larson and S. Vemulapalli in \cite{larson_vemulapalli_2024_planeBN}. 
	\vskip 1 em 
	There they show that, calling $\Delta$ the directrix of $\mathbb F_m$, $\alpha:C\ra \Proj$ a general point in the image of $\mc M_g(\mathbb F_m,(k,a))^{bir}\ra \mc H_{g,k}$ and defining for each $\mbd e,\mbd f\in \mathbb Z^k$
	\[U^{\mbd e,\mbd f}(C):=\{\mcb L\in Pic(C)\ |\ \alpha_*\mcb L\cong \mcb O_{\Proj}(e_1)\oplus\dots\oplus\mcb O_\Proj(e_k),\ \alpha_*(\mcb L(\Delta))\cong \mcb O_{\Proj}(f_1)\oplus\dots\oplus\mcb O_\Proj(f_k)\}\]
	then $U^{\mbd e,\mbd f}(C)$ is non empty if and only if the following inequalities are satisfied:
	\begin{enumerate}[label=$\roman*)$]
		\item for each $i$, $f_i\geq e_i$;
		\item for each $i\leq k-1$, $f_i\geq e_{i+1}-m$;
		\item $\sum_{i\leq k} (f_i-e_i)=a$
	\end{enumerate}
	In which case, denoting by $\mcb O_\Proj(\mbd e):=\mcb O_{\Proj}(e_1)\oplus\dots\oplus\mcb O_\Proj(e_k)$ and by $u(\mbd e):=h^1(\mc End(\mcb O_\Proj(\mbd e)))$, one has that 
	\begin{equation}\label{eq larson}
		\mathrm{dim}(U^{\mbd e,\mbd f}(C))=g-u(\mbd e)-u(\mbd f)+h^1(\Proj, \mc Hom(\mcb O_\Proj(\mbd e),\mcb O_\Proj(\mbd f))\otimes (\mcb O_\Proj\oplus \mcb O_\Proj(m)))
	\end{equation}
	Note that, unlike what happens for Hurwitz-Brill Noether general curves (see \cite{LHk}), these loci do not have the expected dimension (i.e. the expected dimension of the space of deformations of $\mcb O_\Proj(\mbd e)$ and $\mcb O_\Proj(\mbd f)$ which would be:
	$g-u(\mbd e)-u(\mbd f)$).
	\vskip 1 em
	Notice that, if $\nu:C\ra \Gamma\subseteq \mathbb F_m$ is a birational map from a smooth genus $g$ curve to a curve of class $(k,a)$ and $\mcb L$ is a line bundle on $C$, calling $\alpha:C\ra \Proj$ the map given by composing $\nu$ with $\pi$ and $\mbd e,\ \mbd f$ such that $\alpha_*\mcb L\cong \mcb O_\Proj(\mbd e)$ and $\alpha_*(\mcb L(\Delta))\cong \mcb O_\Proj(\mbd f)$, $\mbd e $ and $\mbd f$ still have to satisfy $i)$, $ii)$ and $iii)$.\\
	Indeed, $iii)$ follows from $\mathrm{deg}(\mcb O_C(\Delta)):=\mathrm{deg}(\mcb O_\Gamma(\Delta))=a$ and $i)$ from $h^0(C,\mcb O_C(\Delta))>0$.
	For $ii)$ note that, denoting $\nu_*\mcb L(\Delta)$ by $\mcb M$, $R^1\pi_*\mcb M=R^1\alpha_*(\mcb L(\Delta))=0$. Thus Proposition $1.1$ of \cite{Stromme1987} gives rise to the following exact sequence:
	\[0\ra \pi^*\mcb O_\Proj(\mbd e)(-1,0)\overset{\psi}{\ra} \pi^*\mcb O_\Proj(\mbd f)\ra \mcb M\ra 0\]
	Now, $\psi$ is represented by a pencil of $k\times k$ matrices with entries forms of degree $f_i-e_j$. As noted in Lemma $3.1$ of \cite{larson_vemulapalli_2024_planeBN}, if $ii)$ were to fail, $\mathrm{det}(\psi)$ would be reducible. Since $\mcb M=\mathrm{Coker}(\psi)$ is supported on $\Gamma$ irreducible, this cannot be the case.\\
	Hence for $U^{\mbd e,\mbd f}(C)$ to be non empty for any point $\alpha:C\ra\Proj$ in the image of the map $\mc M_g(\mathbb F_m,(k,a))^{bir}\ra \mc H_{g,k}$ we still need $i),ii)$ and $iii)$ to hold.
	\vskip 1 em
	Since our methods allow us to also compute the splitting type of $\mcb O_C(\Delta)$ also when $g\neq p_a(C)$, we decided to check whether, at least when $e_1+\dots+e_{k-1}=1-g-k$, the above formula still held. Since $\mcb O_C$ is the unique line bundle on $C$ having a global section and degree $0$, this {translates to showing:} 
	\begin{cor}\label{cor 1}
		With the assumptions of Theorem \ref{thm 1}, if $\delta\leq \binom{k-1}{2}m+(a+1)(k-2)$, $\mcb O_C$ satisfies:
		\[h^1(End(\alpha_*\mcb O_C))+h^1(End(\alpha_*\mcb O_C(\Delta))=g+h^1(\mc Hom(\alpha_*\mcb O_C,\alpha_*\mcb O(\Delta))\otimes (\mcb O_{\Proj}\oplus \mcb O_{\Proj}(m)))\]
	\end{cor}
	This could hint at the fact that, at least for small $p_a(k,a)-g$, equation $(\ref{eq larson})$ should still give the correct dimension of the Brill Noether loci for general curves in the image of the forgetful map $\mc M_g(\mathbb F_m,(k,a))^{bir}\ra \mc M_g$. 
	\vskip 2 em
	By Theorem \ref{thm 1}, looking at normalisations of general curves in $\mathbb F_m$ of class $(k,a)$ and geometric genus $g$ gives a very limited selection of vectors $(e_1,\dots, e_{k-1})\in \mathbb N^{k-1}$ arising as scrollar invariants of such covers.
	\vskip 1 em
	Hoping to find whether the study of normalisations of nodal curves in $\mathbb F_m$ could provide some insight on Conjecture \ref{VV conj}, calling $\mc H_{g,(k,a)}$ the image of $\mc M_g(\mathbb F_m,(k,a))^{bir}\ra \mc H_{g,k}$, we then decided to look at curves that lie in some special subloci of $\mc H_{g,(k,a)}$.
	\vskip 1 em
	More precisely, let $C_1,\dots, C_u$ be general sections of $\pi:\mathbb F_m\ra \Proj$ and let for each $i\leq u$, $S_i\subseteq C_i$ be a general subset of $s_i$ points with $s_1\geq\dots\geq s_u$. Our goal will be to determine the scrollar invariants of those degree $k$ covers $\alpha:C\ra \Proj$ that factor as:
	\[
	\begin{tikzcd}
		C \arrow[r, "\nu"] \arrow{dr}[swap]{\alpha} & \Gamma\subseteq \mathbb F_m \arrow[d,"\pi"]\\
		& \Proj
	\end{tikzcd}
	\]
	where $\Gamma$ is a nodal curve having $S:=S_1\cup\dots \cup S_u$ as its singular locus. We first fix some notations. 
	\vskip 1 em
	Let $\sigma_k(i,j,\ell):=\sum_{i< t\leq j}\left(\mathrm{Max}\{(k-i)m+a-1-\ell,0\}-s_t\right)$ and define $n_1$ as the \nobreakdash{minimum} of $\left\{\ell \in \mathbb N\ |\ {\sigma_k(0,j,\ell)<0} \text{ for some $j\leq u$}\right\}$. Let $0=i_0\leq \dots\leq i_u$ so that each $i_t$ is the maximum between $i_{t-1}$ and $\mathrm{Min}\{j>i_{t-1}\ |\ \sigma_{k-i_{t-1}}(i_{t-1},j,n_1)<0\}$, let
	$r_1:=\mathrm{Max}\{i\ |\ n_1\geq im+a\}$ and define $\delta_1:=-\sum_{t\leq u}\sigma_{k-i_{t-1}}(i_{t-1},i_t,\ell)$.
	\begin{thm}\label{main thm}
		With the above notations, the scrollar invariants of the cover $\alpha:C\ra \Proj$ are
		\[
		e_i =\begin{cases}
			im+a & \text{if $i\leq r_1$}\\
			n_1 & \text{if  $r_1<i\leq \delta_1+r_1$}
		\end{cases}\]
		If $i>\delta_1+r_1$, one may compute the $e_i$-s by computing the last $k-1-\delta_1-r_1$ scrollar invariant of a curve of class $(k-i_u,a)$ having as singular locus $S_{i_u+1}\cup \dots \cup S_u$.
	\end{thm}
	\vskip 1 em
	Notice that, since through any $m+1$ points of $\mathbb F_m$ there exists a section of $\pi$ passing through them, we may assume that $s_i\geq m+1$ for each $i<u$.
	\vskip 1 em
	Notice also that, once one has proven that a general point $f:C\ra \Gamma\subseteq \mathbb F_m$ of $\mc M_g(\mathbb F_m,(k,a))^{bir}$ has $\Gamma$ nodal with $\mathrm{Sing}(\Gamma)$ general in $\mc Hilb^\delta(\mathbb F_m)$, we may obtain \ref{thm 1} as a corollary of Theorem \ref{main thm} by taking $s_i\leq m+1$ for each $i$.
	\vskip 1 em
	This theorem, although with different techniques, also generalises a result by M. Coppens (Theorem $1$ in \cite{coppens2021_scrollar}), where he establishes the result when $m=0$ and $u\leq k-1$.
	\vskip 2 em
	The last part of this note is then devoted to showing that, under some assumptions on $a$, such $k$-gonal curves actually exist.
	Again we first fix some notation. Let $\delta:=p_a(k,a)-g$ and define $\phi:\mc M_g(\mathbb F_m,(k,a))^{bir}\ra \mc Hilb^\delta(\mathbb F_m)$ to be the map that assigns to a morphism $\nu:C\ra \Gamma\subseteq \mathbb F_m$ the scheme of length $\delta$ $\mathrm{Sing}(\Gamma)$. Given $s_1\geq \dots\geq s_u$ such that $s_1+\dots+s_u=\delta$, $s_{u-1}\geq m+1$, consider:
	\[\mc H^{s_1,\dots,s_u}_{(1,0)}:=\{(Z,C_1,\dots,C_u)\in \mc Hilb^\delta(\mathbb F_m)\times |\mcb O_{\mathbb F_m}(1,0)|^u\ |\ Z\cap C_i\in \mc Hilb^{s_i}(C_i)\}\]
	and let $\mc M_g^{s_1,\dots, s_u}(\mathbb F_m,(k,a))\subseteq \mc M_g(\mathbb F_m,(k,a))^{bir}$ be the image under the second projection of $\mc H^{s_1,\dots,s_u}_{(1,0)}\times_{\mc Hilb^\delta(\mathbb F_m)}\mc M_g(\mathbb F_m,(k,a))^{bir}$.
	\begin{thm}\label{final thm}
		Let $k\geq 3$. Then, with the above notation, for any $a> 2(\big\lceil\frac{2u}{k}\big\rceil+1)s_1$, the moduli space
		$\mc M_g^{s_1,\dots, s_u}(\mathbb F_m,(k,a))$ is non empty. Moreover, a general point of $\mc M_g^{s_1,\dots, s_u}(\mathbb F_m,(k,a))$ corresponds to a map $f:C\ra \Gamma\subseteq \mathbb F_m$, where $\Gamma\in |\mcb O_{\mathbb F_m}(k,a)|$ is an irreducible nodal curve.
	\end{thm}
	If one sets $m=0$, the above result improves Lemma $1$ in \cite{bal02} both by allowing $u\geq k-1$ and by showing the existence of a nodal curve of class $(k,a)$ in $\Proj\times \Proj$ having singular locus as prescribed above with $a<2us_1+2$.
	\vskip 1 em
	This improvement then allows us to also establish the following result, which proves what E. Ballico conjectured in Remark $3$ of \cite{bal02}.
	\begin{cor}\label{coro P^1}  
		Let $d\in \mathbb Z$ and let $g> 6d(k-1)$. Then for any sequence
		$e_1\leq\dots\leq e_{k-1}=e_1+d$ such that $\sum_{i\leq k-1}e_i=g+k-1$, there exists a degree $k$ cover $\alpha:C\ra \Proj$ from a smooth genus $g$ curve having $e_1,\dots,e_{k-1}$ as scrollar invariants.
	\end{cor}
	Combining Theorem \ref{main thm} with Theorem \ref{final thm} with $m>0$ we were also able to find some more points of $\mc P_k$ that are not given by Corollary \ref{coro P^1}. Unfortunately, we were unable to find a way of getting subsets of positive measure.
    \begin{akn} This research has been carried out as part of the DFG Research Training Group 2965 "From geometry to numbers" involving Humboldt-Universität zu Berlin and Leibniz Universität Hannover. This work has also been funded by the Deutsche Forschungsgemeinschaft (DFG, German Research
    Foundation) under Germany's Excellence Strategy – The Berlin Mathematics
    Research Center MATH+ (EXC-2046/1, EXC-2046/2, project ID: 390685689).\\ 
    I wish to thank my advisor Gavril Farkas for his guidance and support. I would also like to thank Bogdan Carasca, Andreas Kretschmer and Irene Spelta for many insightful conversations.
    \end{akn}

	\section{BACKGROUND ON HIRZEBRUCH SURFACES}
	By the projective bundle formula, (Theorem 3.3 in \cite{fulton}), the Chow ring of $\mathbb F_m$ is isomorphic to
	\begin{equation}\label{Chow ring}
		\frac{\mathbb Z[\zeta,\xi]}{(\xi^2,\zeta^2-m\xi\zeta)}
	\end{equation}
	Thus, with this notation, the class of the directrix of $\mathbb F_m$ is $[\Delta]:=\zeta-m\xi$. With this notation, its canonical class is $K=(m-2)\xi-2\zeta$. By adjunction it then follows that the virtual genus of a divisor $D$ of type $(k,a)$ is $p_a(D)=\binom{k}{2}m+(k-1)(a-1)$.
	\begin{oss}\label{rmk 2}
		Since $\pi_*(\mcb O_{\mathbb F_m}(n))\cong \mc Sym^n(\mc E_m)$, we may identify global sections of $\mcb O_{\mathbb F_m}(k,a)$ with those of $\mc Sym^k(\mc E_m)\otimes \mcb O_{\Proj}(a)$. Choosing $u_0$, $u_1$ as relative coordinates on $\mathbb F_m$ and $t_0$, $t_1$ as coordinates on $\Proj$, one may identify global sections of $\mcb O_{\mathbb F_m}(k,a)$ with polynomials of the form:
		\[\sum_{im+a\geq 0} f_i(t_0,t_1)u_0^iu_1^{k-i}, \qquad \text{where $f_i\in H^0(\Proj,\mcb O_{\Proj}(im+a))$}\]
		Under this identification, the directrix corresponds to the vanishing locus of $u_0$, while the other sections of $\pi$ correspond to vanishing loci of polynomials of the form:
		\[{u_1-h(t_0,t_1)u_0},\ \ \ \text{ with $\mathrm{deg}(h)=m$}\]
	\end{oss}
	\begin{oss}
		Any irreducible effective divisor $D$ of type $(k,a)$ comes endowed with a degree $k$ cover $f:D\ra \Proj$ given by restricting the projection $\pi:\mathbb F_m\ra \Proj$.
	\end{oss}
	Finally, we shall denote by ${\mc M}_{g}(\mathbb F_m,(k,a))$ the moduli stack parametrising morphisms\\ $f:C\ra \mathbb F_m$ such that $f_*[C]=(k,a)\in A^1(\mathbb F_m)$. Notice that, calling $\mc S\subseteq |\mcb O_{\mathbb F_m}(k,a)|$ the Severi variety parametrising genus $g$ irreducible nodal curves, and calling $\mathfrak C$ the universal curve over $\mc S$, since families of nodal curves are equinormalisable, remembering the normalisation gives a morphism $\mc S\ra \mc M_g$. Calling $\mcb C$ the pullback of the universal family of $\mc M_g$ to $\mc S$ one has a commutative diagram:
	\[\begin{tikzcd}
		\mcb C \arrow[d,] \arrow[r,"\mcb f"]& \mathbb F_m\times \mc S \arrow[d,]\\
		\mathfrak C \arrow[ur,hook,]  \arrow[r,"\mcb p",swap] & \mc S
	\end{tikzcd}\]
	This then induces an injective morphism $\mcb s: \mc S\ra \mc M_g(\mathbb F_m,(k,a))$.
	\begin{oss}
		Notice that, since by Proposition $2.7$ of \cite{christ2023severi} the Severi variety parametrising irreducible curves of genus $g$ and class $(k,a)$ in $\mathbb F_m$ is irreducible of the expected dimension, by Zariski's Theorem (Corollary 5.3 in \cite{christ2023severi}), the same is true for $\mc S$.
	\end{oss}
	Now, as shown for example in the proof of Proposition $2.7$ of \cite{christ2023severi}, locally around each $f:C\ra \mathbb F_m$ in the image of $\mcb s$ one has a morphism to $\ol{\mc S}$ with finite fibres. In particular, there is a unique component of $\mc M_g(\mathbb F_m,(k,a))$ containing the image of $\mcb s$. We shall denote this component by $\mc M_g(\mathbb F_m,(k,a))^{bir}$. Note in passing that, since the expected dimension of $\mc M_g(\mathbb F_m,(k,a))$ is the same as the one for $\mc S$, this component will have the expected dimension.

	\section{INTERPOLATING POINTS IN $\mathbb F_m$}
	Given a nodal curve $\Gamma\subseteq \mathbb F_m$, calling $\nu:C\ra \Gamma$ the normalisation and $f:=\pi\circ \nu$, the scrollar invariants of $C$ are the unique integers $0\leq e_1\leq \dots\leq e_{k-1}$ such that:
	\[\Tsch(f):={f_*\mcb O_C}/{\mcb O_{\Proj}}\cong \bigoplus_{i=1}^{k-1}\mcb O_{\Proj}(e_i)\]
	In particular, calling $\mc A:=f^*\mcb O_{\Proj}(1)$, $e_i=\mathrm{Min}\{n\in \mathbb N\ |\ h^0(C,\mc A^{\otimes n})-h^0(C,\mc A^{\otimes n-1})>i\}$. By Riemann-Roch, this is equivalent to \[e_i=\mathrm{Min}\left\{n\in \mathbb N\ |\ h^0(C,\omega_C\otimes\mc A^{\otimes 1-n})-h^0(C,\omega_C\otimes \mc A^{\otimes -n})<k-i\right\}\]
	Now, since $\Gamma$ is a nodal curve, if one denotes by $S$ its singular locus, one has that the conductor ideal of $\mcb O_\Gamma$ in $\mcb O_C$ is induced by the ideal sheaf of $S$ in $\Gamma$, which we shall henceforth denote by $\mc I_{S,\Gamma}$. Hence $\nu_*\omega_C\cong \omega_\Gamma\otimes \mc I_{S,\Gamma}$ and, denoting by $\mc I_S$ the ideal sheaf of $S$ in $\mathbb F_m$, one has that
	\begin{equation}\label{h0 interp}
		0\ra \mcb O_{\mathbb F_m}(-2,m-2-n)\ra \mc I_S(k-2,m+a-2-n)\ra \nu_*(\omega_C\otimes \mc A^{\otimes -n})\ra 0
	\end{equation}
	By Leray's spectral sequence, $h^1(\Proj,\mcb O_{\mathbb F_m}(-2,m-2-n))=0$ for each $n\geq 0$. Thus \[H^0(C,\omega_C\otimes \mcb A^{\otimes -n})\cong H^0(\mathbb F_m,\mc I_S(k-2,a+m-2-n))\]
	If one now calls $f^S_{k,a}(n):=h^0(C,\omega_C\otimes \mcb A^{\otimes 1-n})-h^0(C,\omega_C\otimes \mcb A^{\otimes -n})$, one has that
	\[e_i=\mathrm{Min}\left\{n\in \mathbb N\ |\ f^S_{k,a}(n)<k-i\right\}\]
	The problem of computing the scrollar invariants of nodal curves in $\mathbb F_m$ having singular locus $S$ can thus be reduced to that of finding how many independent conditions $S$ imposes on the space of forms of bidegree $(k-2,a+m-2-n)$ on $\mathbb F_m$. If the points are general one readily has:
	\begin{lem}\label{interpolating general points}
		Let $S$ be a collection of $s$ general points of $\mathbb F_m$. Then, if $s\leq h^0(\mathbb F_m,\mcb O_{\mathbb F_m}(k,a))$, passing through $S$ imposes $s$ independent conditions on $H^0(\mathbb F_m,\mcb O_{\mathbb F_m}(k,a))$. 
	\end{lem}
	\begin{proof}
		Let $D=(p_1,\dots,p_s)\in (\mathbb F_m)^s$ and define:
		\[ev_D: H^0(\mathbb F_m,\mcb O_{\mathbb F_m}(k,a))\ra \mcb k^{\oplus s}\ |\ f\mapsto (f_{p_i})_{i=1}^s\]
		Let $\mc K_s$ be the incidence correspondence:
		\[\{(\mbd v, D)\in \Pb H^0(\mathbb F_m,\mcb O_{\mathbb F_m}(k,a)\times (\mathbb F_m)^s\ | \mbd v\in Ker(ev_D)\}\] and let $p:\mc K_s\ra (\mathbb F_m)^s$ be the second projection. Notice that $p$ is clearly surjective. Since the dimension of a fibre is always greater or equal than $h^0(\mathbb F_m,\mcb O_{\mathbb F_m}(k,a))-s$, by semicontinuity of the fibre dimension, we need only to exhibit one point $D\in (\mathbb F_m)^s$ for which $ev_D$ has maximal rank.\\
		Now choose $\{t_0^\ell t_1^{jm+a-\ell}u_0^ju_1^{k-j}\}_{j,\ell}$ as basis for $H^0(\mathbb F_m,\mcb O_{\mathbb F_m}(k,a))$. With respect to this basis $ev_D$ is represented by a matrix $E$ whose $i$-th row is given by:
		\[(t_0(p_i)^\ell t_1(p_i)^{jm+a-\ell}u_0(p_i)^ju_1(p_i)^{k-j})_{j,\ell}\]
		Thus, saying that $ev_D$ is surjective is equivalent to asking that there exists a non-zero $s\times s$ minor of $E$. Now notice that, denoting by $\Delta_{\mathbb F_m}$ the fan giving rise to $\mathbb F_m$, since by \cite{fultonToric} $(\mathbb F_m)^s$ is a toric variety given by the following fan $\prod_{i=1}^s\Delta_{\mathbb F_m}$, by \cite{coxRing}, $\{t_0(p_i),\ t_1(p_i),\ u_0(p_i),\ u_1(p_i)\}_{i=1}^s$ are algebraically independent in the Cox Ring of $(\mathbb F_m)^s$. In particular, since no two monomials on the same row are the same, the vanishing of any $s\times s$ minor of $E$ cuts out a divisor on $(\mathbb F_m)^s$. This implies that, as long as $s\leq h^0(\mcb O_{\mathbb F_m}(k,a))$, there is an open $U$ in $(\mathbb F_m)^s$ such that for each $D\in U$, $ev_D$ is surjective.
	\end{proof}
	\begin{oss}\label{good remark}
		Notice that the previous proof shows more generally that for a general tuple of points $D=(p_1,\dots,p_s)\in (\mathbb F_m)^s$, the map $ev_D$ has maximal rank. 
	\end{oss}
	If instead the points are contained in some sections of $\mathbb F_m\ra \Proj$, one may generalise the techniques described in the proof of Lemma $3$  in \cite{coppens2021_scrollar} to prove the following lemma. Even though the techniques are very similar to those of \cite{coppens2021_scrollar}, we include the proof as the lemma will be useful in the proof of Theorem \ref{final thm}.
	\begin{lem}\label{easy h0}
		Let $C_0,\dots,C_k$ be distinct sections of $\pi:\mathbb F_m\ra \Proj$ and let $\{D_i\subseteq C_i\}_{i=0}^k$ be a collection of effective divisors of degrees $s_0,\dots,s_k$ such that for each $i$, \[s_i\leq \mathrm{Min}\{s_{i+1},\mathrm{Max}\{im+a+1,0\}\}\] and for each $i,j$, $D_i\cap C_j=\emptyset$. Then, calling $Z$ the $0$-dimensional subscheme of $\mathbb F_m$ given by their union and $\mc I_Z$ its ideal sheaf, 
		\[h^0(\mathbb F_m, \mc I_Z(k,a))=h^0(\mathbb F_m,\mcb O_{\mathbb F_m}(k,a))-\sum_{i=0}^ks_i\]
	\end{lem}
	\begin{proof}
		Note first that, since each of the $C_i$ is a smooth curve, there is no ambiguity when talking about a length $n$ subscheme of $C_i$ supported at a given point. Hence, this lemma also computes the number of parameters of curves satisfying a tangency condition with respect to $C_i$ at a given point.
		\vskip 1 em
		Since global sections of $\mcb O_{\mathbb F_m}(k,a)$ correspond to polynomials of the form $p=\sum_{i=0}^kf_i(t_0,t_1)u_0^iu_1^{k-i}$ with $f_i$ homogeneous of degree $im+a$ (and polynomials of negative degree corresponding to $0$), if $-im\leq a < -(i-1)m$, sections of $H^0(\mathbb F_m,\mcb O_{\mathbb F_m}(k,a))$ are of the form $u_0^{i}q$ with ${q}$ a global section of $\mcb O_{\mathbb F_m}(k-i,im+a)$. Since, by assumption, one also has that $s_j=0$ for each $j<k-i$, we need only to prove the statement for $a\geq 0$.
		\vskip 1 em
		The proof will proceed by induction on $k$. If $k=0$, any divisor in $|\mcb O_{\mathbb F_m}(0,a)|$ corresponds to a sum of $a$ fibres of $\pi$. Hence, since $s_0\leq a+1$, asking for it to pass through $s_0$ points imposes independent conditions on $H^0(\mathbb F_m,\mcb O_{\mathbb F_m}(0,a))$.\\
		Note that, even when some of the points appear with multiplicity $>1$, they still impose independent conditions on the lines as vanishing with multiplicity $n$ on a given point requires $n$ fibres to coincide.
		\vskip 1 em
		Now let $k\geq 1$ and assume that the statement holds for $k-1$. Define ${\nu(Z)}:=\mathrm{Max}(i\ |\ s_i\neq 0)$. If $\nu(Z)=0$, there is nothing to prove. Assuming now for the statement to hold for each $\nu'<\nu(Z)$, let $Z'$ be the susbcheme defined by the union of the $D_i$-s with $i<k$. By construction, $\nu(Z')<\nu(Z)$. Hence, by induction, one has that $h^0(\mathbb F_m,\mc I_{Z'}(k,a))=h^0(\mathbb F_m,\mcb O_{\mathbb F_m}(k,a))-\sum_{i<k}s_i$. Choosing now $\Gamma$ as the union of $Z'$ and $C_k$, the following sequence is exact:
		\[0\ra \mc I_\Gamma(k,a)\ra \mc I_{Z'}(k,a)\ra \mcb O_{C_k}(k,a)\ra 0\]
		Now if $H^1(\mathbb F_m,\mc I_\Gamma(k,a))\neq 0$, vanishing on $C_k$ would impose less than $h^0(C_k,\mcb O_{C_k}(k,a))$ conditions on $H^0(\mathbb F_m,\mc I_{Z'}(k,a))$. In particular, $h^0(\mathbb F_m,\mc I_{Z'}(k-1,a))\geq h^0(\mathbb F_m,\mc I_{Z'}(k,a))-km-a$.
		However, by induction hypothesis, one has that:
		\[h^0(\mathbb F_m,\mc I_{Z'}(k-1,a))= h^0(\mathbb F_m,\mcb O_{\mathbb F_m}(k-1,a))-\sum_{i<k}s_i= h^0(\mathbb F_m,\mcb O_{\mathbb F_m}(k,a))-km-a-1-\sum_{i<k}s_i\] 
		Thus the map $H^0(\mathbb F_m,\mc I_{Z'}(k,a))\ra H^0(C_k,\mcb O_{C_k}(k,a))$ must be surjective. In particular, for the result to hold it is sufficient to prove that passing through $s_k$ points imposes independent conditions on $H^0(C_k,\mcb O_{C_k}(k,a))$. Now this immediately follows from the condition on $s_k$ as global sections of $\mcb O_{C_k}(k,a)$ correspond to homogeneous polynomials of degree $km+a$ in two variables.		
	\end{proof}
	If we allow ourselves to instead choose general sections and general points in them, we may drop the assumptions on both the number of sections and the bound on the $s_i$. Indeed we have:
	\begin{lem}\label{good lem}
		Let $C_1,\dots,C_u$ be general sections of $\pi:\mathbb F_m\ra \Proj$ and let for each $i=1,\dots, u$, $S_i$ be a set of $s_i$ general points of $C_i$ with $s_1\geq\dots \geq s_u>0$. Then, if for each $i=1,\dots,u$,
		\begin{equation}\label{star}
			\sum_{j\leq i} s_j\leq \sum_{j\leq i}\mathrm{Max}\{(k-j)m+a+1,0\}
			\tag{\textasteriskcentered }
		\end{equation}
		passing through the $S_i$ imposes independent conditions on $|\mcb O_{\mathbb F_m}(k,a)|$.
	\end{lem}
	\begin{proof}
		As noted in Remark \ref{rmk 2}, sections of $\pi:\mathbb F_m\ra \Proj$ correspond to polynomials of the form $u_1-hu_0$ with $h\in H^0(\Proj,\mcb O_\Proj(m))$. To give general sections $C_1,\dots,C_u$ of $\pi$ is hence equivalent to giving general elements $h_1,\dots,h_{u}\in H^0(\Proj,\mcb O_{\Proj}(m))$. Let now for each $i=1,\dots u$, $S_i=\{p_{i,1},\dots p_{i,s_1}\}$ and let $T_i:=\pi(S_i)=\{q_{i,1},\dots,q_{i,s_1}\}$. 
		Define \[ev_\mbd q^{\mbd h} : V_n:=\prod_{i=0}^{k} H^0(\mcb O_{\Proj}(im+a))\ra \bigoplus_{j=l}^{k}\mcb k^{\oplus s_j}\] \[(f_i)_{i=0}^{k}\xmapsto{\qquad} \left(\left( \sum_{i=0}^{k} f_i(q_{j,t})h_j^{k-i}(q_{j,t})\right)_{t=1}^{s_j}\right)_{j=1}^{u}\]
		By construction, $Ker(ev_\mbd q^{\mbd h})$ is the subspace of $H^0(\mathbb F_m,\mcb O_{\mathbb F_m}(k,a))$ of global sections vanishing over each of the $S_i$. Hence, if one defines $\delta$ to be the sum of the $s_i$ and
		\[\mc K_n:=\{(\mbd h, \mbd q, \mbd v)\in H^0(\Proj,\mcb O_{\Proj}(m))^{u}\times (\Proj)^\delta\times V_n |\ \mbd v\in Ker(ev_\mbd q^\mbd h)\}\]
		the proof of the lemma reduces to finding the dimension of a general fibre of the projection
		\[p_1:\mc K_n\ra H^0(\Proj,\mcb O_{\Proj}(m))^u\times (\Proj)^\delta\]
		Notice moreover that, by semi-continuity of fibre dimension, once one finds an $\mbd h$ and a $\mbd q$ such that $ev_\mbd q^\mbd h$ has maximal rank, the dimension of the general fibre is the expected one.
		\vskip 1 em
		Notice also that, choosing $\{t_{0,j},t_{1,j}\}_{j=1}^\delta$ as coordinates on $(\Proj)^\delta$, with respect to the standard monomial basis on $H^0(\mathbb F_m,\mcb O_{\mathbb F_m}(k,a))$, $ev_\mbd q^\mbd h$ is represented by the following matrix
		\[\begin{tikzpicture}[baseline=(m.center)]
			\node (m) {$\left(\begin{array}{cccccc}
					t_{0,1}^{km+a} & \dots & t_{1,1}^{km+a} & h_1t_{0,1}^{(k-1)m+a} & \dots & h_1^{k}t_{1,1}^{a}\\
					\dots & \dots & \dots & \dots & \dots & \dots\\
					t_{0,s_1}^{km+a} & \dots & t_{1,s_1}^{km+a} & h_1t_{0,s_1}^{(k-1)m+a} & \dots & h_1^{k}t_{1,s_1}^{a}\\
					&  &  & & & \\
					\dots & \dots & \dots & \dots & \dots & \dots\\
					&  &  & & & \\
					t_{0,\delta-s_u+1}^{km+a} & \dots & t_{1,\delta-s_u+1}^{km+a} & h_{u}t_{0,\delta-s_u+1}^{(k-1)m+a} & \dots & h_{u}^{k}t_{1,\delta-s_u+1}^{a}\\
					\dots & \dots & \dots & \dots & \dots & \dots\\
					t_{0,\delta}^{km+a} & \dots & t_{1,\delta}^{km+a} & h_{u}t_{0,\delta}^{(k-1)m+a} &\dots & h_{u}^{k}t_{1,\delta}^{a}
				\end{array}\right)$};
			\draw[decorate,decoration={brace,mirror,amplitude=5pt},xshift=-6pt]
			([yshift=-2pt]m.north west)
			-- ([yshift=4pt]$(m.north west)!0.4!(m.south west)$)
			node[midway,left=7pt]{\text{$s_{1}$}};
			\draw[
			decorate,
			decoration={brace,mirror,amplitude=5pt},
			xshift=-6pt
			]
			($(m.north west)!0.65!(m.south west)$)
			-- ([yshift=2pt]m.south west)
			node[midway,left=7pt]{\text{$s_{u}$}};
		\end{tikzpicture}\]
		\normalfont
		The proof will now proceede by induction on $u$. Notice that, by our assumption on the $s_i$-s, if $u=1$, the theorem follows by Lemma
		\ref{easy h0}. Now assume the lemma to hold for each $u'<u$, let's prove it for $u$.\\
		Again by our assumptions on the $s_i$-s, there exists a permutation $\sigma\in \mathrm{Aut}(\{1,\dots,k+1\})$ and a partition of $\{k+2,\dots,u\}$ into $k+1$ subsets $\Lambda_1,\dots,\Lambda_{k+1}$ such that, defining for each $i\leq k+1$, $s'_i:=s_{\sigma(i)}+\sum_{j\in \Lambda_i}s_j$, the $s'_i$-s still satisfy (\ref{star}). 
		\vskip 1 em 
		We may hence assume that $1<u\leq k+1$. Indeed, if $u>k+1$, then, taking $h'_1,\dots,h'_{k+1}$ general in $H^0(\Proj,\mcb O_\Proj(m))$ and $h_j:=h'_\ell$ whenever $j=\sigma(\ell)$ or $j\in \Lambda_\ell$, by induction hypothesis there exists a $\mbd q$ such that $ev_\mbd q^\mbd h$ has maximal rank as, up to a permutation of the rows, the matrix representing $ev_\mbd q^\mbd h$ corresponds to the one representing $ev_\mbd q^{\mbd h'}$.\\
		Notice now that, calling $D_c[s,\ell]:=\mathrm{Diag}\left(\prod_{i=0}^{c-1}(h_c-h_i)(t_{0,j},t_{1,j})_{j=s+1}^\ell\right)$ and $A_i[s,l]$ the matrix whose rows are the standard basis of $H^0({\mathbb F_m},\mcb O_{\mathbb F_m}((k-i)m+a))$ taken in the variables $t_{0,j}, t_{1,j}$ for $s<j\leq \ell$, using elementary column transformations one gets a block matrix in the form:\\
		\begingroup
		\small
		
		\begin{equation}\label{block triang}
			\left(\begin{array}{cccccccc}
				A_{0}[0,s_1] & 0 & 0 & \dots & 0 & 0 & \dots & 0\\
				A_0[s_1,s_1+s_2] & (D_1A_1)[s_1,s_1+s_2] & 0 & \dots & 0 & 0 & \dots & 0\\
				\dots &  \dots & \dots & & \dots  & & &\\
				A_0[\delta-s_{u},\delta] & (D_1A_1)[\delta-s_{u},\delta] & (D_2A_2)[\delta-s_{u},\delta] & \dots & (D_{u}A_{u})[\delta-s_{u},\delta] & 0 & \dots & 0\\
				&  &  & &  & & &\\
				\dots & \dots & \dots & \dots & \dots& & &\\
				&  &  & & & & &\\
				A_0[\delta-s_{u},\delta] & (D_1A_1)[\delta-s_{u},\delta] & (D_2A_2)[\delta-s_{u},\delta] & \dots & (D_{u}A_{u})[\delta-s_{u},\delta] & 0 & \dots & 0
			\end{array}\right)
		\end{equation}
		\endgroup
		\vskip 2em
		\normalfont
		Introduce now some new variables $\lambda_1,\dots,\lambda_{u}$, let $h_i=\lambda_it_0^m$ and choose the lexicographic term ordering on the monomials of $\mcb k[\lambda_1,t_{0,1},t_{1,1},\dots,t_{0,\delta},t_{1,\delta}][\lambda_2,\dots,\lambda_{u}]$. Note that, once one has proven that with this choice of $\mbd h$, calling $E$ the $\delta\times \delta$ submatrix of the above matrix given by the first $\delta $ columns, the coefficient of the monomial of $|E|$ with maximal possible term ordering is not zero, we are done. Note moreover that in the case $u=1$ this is clearly true as $|E|$ is just the determinant of a Vandermonde matrix. Hence we may continue with the proof by induction to not only show that $|E|\neq 0$, but that, with this choice of $h_i$-s, the coefficient of the monomial of $|E|$ with maximal possible term ordering is not zero in $\mcb k[\lambda_1,t_{0,1},t_{1,1},\dots,t_{0,\delta},t_{1,\delta}]$.
		\vskip 1 em
		Let $\delta':=\delta-km-a-1$ and denote by $\mc P_{\delta'}$ the set of all subsets of $\{s_{1}+1,\dots,\delta\}$ having cardinality $\delta'$.\\
		For each $J\in \mc P_{\delta'}$, denote by $E_J$ the submatrix of $E$ given by the elements of the $\delta'$ columns following the first $km+a+1$ having row indices in $J$ and by $A_J$ the one defined by the elements of the first $\delta-\delta'$ columns whose row indices are not in $J$.\\ 
		Calling  $sgn(J):=(-1)^{\delta'(\delta-\delta')+\binom{\delta'+1}{2}+\sum_{j\in J} j}$, by Laplace's expansion formula, one has that:
		\[|E|=\sum_J sgn(J)|A_J||E_J|\]
		Now, for each $i\geq 2$, let $S'_i$ be the subset of $S_i$ given by the first $m_i$ elements of $S_i$ where 
		\[m_i:=\mathrm{Min}\left\{s_i,(k-i+1)m+a+1+\sum_{2\leq \ell<i}[(k-\ell+1)m+a+1-m_l]\right\}\]
		\[r_0:=\mathrm{Max}\left\{j \ |\ \sum_{2\leq i\leq j} m_i<\delta'\right\}\]
		and let $M$ be the set of indices of the $S'_i$ for $i\leq r_0$ together with the first $\delta'-\sum_{2\leq i\leq r_0}m_i$ ones of $S'_{r_0+1}$. By $(\ref{block triang})$, 
		\[
		E_{M}= \mathrm{Diag}\left(\left((h_j-h_0)(t_{0,\ell},t_{1,\ell})\right)_{\ell\in M}\right)E(\mbd h',\mbd q')
		\]
		where $\mbd h'=\{h_{2},\dots,h_{r_0+1}\}$ and $\mbd q'$ is given by the $\delta'$ points in the $S'_i$. Since, by construction, the $\mbd q'$ satisfy (\ref{star}), by induction hypothesis, for generic $\mbd h$ this is non $0$. Being a determinant of a Vandermonde matrix with distinct rows, the same is true for $|A_{M}|$.
		\vskip 1 em
		Now, given $J\in \mc P_{\delta'}$, define $J_i:=\{j\in J\ |\ \sum_{l< i} s_l<j\leq \sum_{l\leq i} s_l\}$ and define a partial ordering $\succcurlyeq$ on $\mc P_{\delta'}$ by setting $J\succcurlyeq J'$ if and only if there is some $2\leq j\leq i_0$ such that $J_j\supseteq J'_j$ and, for each $i<j$, $J_i=J'_i$.
		Let \[\mc P:=\{J\in \mc P_{\delta'}\ |\ \text{for each}\ i,\ \sum_{\ell\leq i} \#J_\ell \leq \sum_{\ell\leq i} ((k-\ell-1)m+a+1)\}\]
		Notice that, by the block structure given in $(\ref{block triang})$, if $J\notin \mc P$, $|E_J|=0$. By induction hypothesis, one also has that for each $J\in\mc P$, $|E_J|\neq 0$. By construction, the $M$ defined above is maximal in $\mc P$ with respect to $\succcurlyeq$.
		\vskip 1 em
		\textit{\bd{Claim:}} For any $i$ such that $\sum_{2\leq \ell\leq i} s_\ell\leq \delta'$, \begin{equation}\label{win}
			\sum_{2\leq \ell\leq i}(s_\ell-m_\ell)\leq (i-1)m
		\end{equation}
		\vskip 1 em
		Indeed, choose $\ell_1,\dots,\ell_t$ to be the indices such that $m_i\neq s_i$, for each $i<\ell_1$, 
		\[\sum_{2\leq \ell\leq i}(s_\ell-m_\ell)=0\] 
		for each $\ell_1\leq i<\ell_2$,
		\[\sum_{2\leq \ell\leq i}(s_\ell-m_\ell)=s_{\ell_1}-m_{\ell_1}:=\sum_{2\leq \ell\leq \ell_1}\left[ s_\ell-((k-\ell+1)m+a+1)\right]\leq km+a+1-s_1\]
		Now though, since $s_{\ell_1}\leq s_1$ and $m_{\ell_1}\neq s_{\ell_1}$ implies $s_{\ell_1}\geq (k-\ell_1+1)m+a+1$, the above implies:
		\[\sum_{2\leq \ell\leq i}(s_\ell-m_\ell)\leq (\ell_1-1)m\leq (\ell-1)m\]
		Now consider $l_j\leq i< l_{j+1}$ with $j>1$. By assumption,
		\[m_{\ell_j}=(k-\ell_j+1)m+a+1+\sum_{2\leq \ell< \ell_j}\left[ (k-\ell+1)m+a+1-m_\ell\right]\]
		\[m_{\ell_{j-1}}=(k-\ell_{j-1}+1)m+a+1+\sum_{2\leq \ell< \ell_{j-1}}\left[ (k-\ell+1)m+a+1-m_\ell\right]\]
		Putting the two together one then has 
		\[m_{\ell_j}=(k-\ell_j+1)m+a+1+\sum_{\ell_{j-1}< \ell< \ell_j}\left[ (k-\ell+1)m+a+1-s_\ell\right]\]
		Thus:
		\[s_{\ell_j}-m_{\ell_j}= \sum_{\ell_{j-1}< \ell\leq  \ell_j}\left[ s_\ell-((k-\ell+1)m+a+1)\right]\]
		Putting everything together one then has
		\[\sum_{u\leq j}(s_{\ell_u}-m_{\ell_u})=\sum_{2\leq \ell\leq \ell_j}\left[ s_\ell-((k-\ell+1)m+a+1)\right]\leq km+a+1-s_1\leq (\ell_j-1)m\]
		Thus concluding the proof of claim.
		\vskip 1 em
		Notice also that, \[m_{r_0}=\delta'-\sum_{2\leq \ell\leq r_0}m_\ell<s_{r_0}\leq s_1\] and $M_j=\emptyset$ for each $j>r_0$.\\
		Notice that, by definition and our choice of term order on $\mcb k[\lambda_1,t_{0,1},t_{1,1},\dots,t_{0,\delta},t_{1,\delta}][\lambda_2,\dots,\lambda_{u}]$, the maximal elements of $\mc P$ with respect to $\preceq$ are precisely the indices giving the summands of $|E|$ having the monomial in $\lambda_2,\dots,\lambda_{u}$ of maximal term order. Hence, to prove the lemma, we may restrict our attention to these ones.
		\vskip 1 em
		Let for each $i=1,\dots,m_{r_0}$ $j_i$ be the $i$-th element of $M_{r_0}$ and restrict $|E|$ to $t_{1,j_i}=t_{1,i}$ for each $i=1,\dots,m_{r_0}$ and to $t_{0,i}=1$ for each $i$. This restriction sets to zero every summand in the Laplace expansion of $|E|$ indexed by a $J$ that does not contain $M_{r_0}$. Now, for each $i<r_0$, let ${N_i=\{l_1,\dots, l_{s_i-m_i}\}}$ be the complement of $M_i$ in $\{\sum_{\ell<i} s_\ell+1,\dots,\sum_{\ell\leq i} s_\ell\}$ and define 
		\[p:=\prod_{i<r_0} \prod_{j\leq s_i-m_i} t_{\ell_j}^{km+a+1-\sum_{l<i}\#N_l-j}\]
		\begingroup
		\small
		\[\begin{tikzpicture}[baseline=(m.center)]
			\node (m) {$\left(\begin{array}{cccccccc}
					1 & \dots & t_{1}^{(k-1)m+a+x_2} & \dots & t_{1}^{km+a} & 0 & \dots \\
					\dots & \dots & \dots  & \dots & \dots & \dots & \dots\\
					1 & \dots & t_{s_1}^{(k-1)m+a+x_2} & \dots & t_{s_1}^{km+a} & 0 & \dots \\
					1 & \dots & t_{s_1+1}^{(k-1)m+a+x_2} & \dots & t_{s_1+1}^{km+a} & \lambda_2-\lambda_1 & \dots \\
					\dots & \dots & \dots  & \dots & \dots & \dots & \dots\\
					1 & \dots & t_{s_1+m_2}^{((k-1)m+a+x_2} & \dots & t_{s_1+m_2}^{km+a} & \lambda_2-\lambda_1 & \dots \\[4pt]
					1 & \dots & t_{s_1+m_2+1}^{(k-1)m+a+x_2} & \dots & t_{s_1+m_2+1}^{km+a} & \lambda_2-\lambda_1 & \dots \\
					\dots & \dots & \dots  & \dots & \dots & \dots & \dots\\
					1 & \dots & t_{s_1+s_2}^{(k-1)m+a+x_2} & \dots & t_{s_1+s_2}^{km+a} & \lambda_2-\lambda_1 & \dots \\[2pt]
					\dots & \dots & \dots  & \dots & \dots & \dots & \dots\\
				\end{array}\right)$};
			\draw[decorate,decoration={brace,mirror,amplitude=5pt},xshift=-6pt]
			([yshift=-3pt]m.north west)
			-- ([yshift=4pt]$(m.north west)!0.33!(m.south west)$)
			node[midway,left=7pt]{\text{$s_{1}$}};
			\draw[decorate,decoration={brace,mirror,amplitude=5pt},xshift=-6pt]
			([yshift=-43pt]m.north west)
			-- ([yshift=-40pt]$(m.north west)!0.27!(m.south west)$)
			node[midway,left=7pt]{\text{$m_2$}};
			\draw[decorate,decoration={brace,mirror,amplitude=5pt},xshift=-6pt]
			([yshift=-80pt]m.north west)
			-- ([yshift=-68pt]$(m.north west)!0.33!(m.south west)$)
			node[midway,left=7pt]{\text{$s_{2}-m_2$}};
			\draw[red, thick, rounded corners]
			([xshift=12pt,yshift=-1 pt]m.north west) rectangle
			([xshift=188pt,yshift=-38pt]m.north west);
			\draw[red, thick, rounded corners]
			([xshift=12pt,yshift=-78 pt]m.north west) rectangle
			([xshift=188pt,yshift=-120pt]m.north west);  
			\draw[blue, thick, rounded corners]
			([xshift=-70pt,yshift=55pt]m.south east) rectangle
			([xshift=-10pt,yshift=92pt]m.south east);
			\draw[blue, thick, rounded corners]
			([xshift=-70pt,yshift=2pt]m.south east) rectangle
			([xshift=-10pt,yshift=13pt]m.south east);
			\draw[orange, thick]
			([xshift=150pt,yshift=-93pt]m.north west) -- ([xshift=183pt,yshift=-93pt]m.north west);
			\draw[orange, thick]
			([xshift=62pt,yshift=-117pt]m.north west) -- ([xshift=87pt,yshift=-117pt]m.north west);
			
			\node[below=6pt of m, align=center, font=\small] (caption)
			{Fig.~1.\ Calling $x_2:=m-(s_2-m_2)+1$, the figure shows how the procedure described};
			\node[below=16pt of m, align=center, font=\small] (caption)
			{above selects one summand in the expansion of $|E|$};
		\end{tikzpicture}\]
		\endgroup
		Putting together \ref{block triang} and \ref{win} one gets that $|A_M||E_M|$ is the unique summand of $|E|$ multiplying the monomial in $\lambda_2,\dots,\lambda_{u}$ of maximal term order having as coefficient the monomial $p$ and having index containing $M_{r_0}$. Since no other summand in the Laplace expansion of $|E|$ restricted to $t_{0,i}=1$ for each $i$, and to $t_{1,j_i}=t_{1,i}$ for each $i=1,\dots,m_{r_0}$ can cancel out this monomial, this concludes the proof of the lemma.\\
	\end{proof}
	Notice that, using the same notation as in Lemma \ref{good lem}, the structure of the matrix in $(\ref{block triang})$ may more generally be used to compute how many conditions passing through the points in the $S_i$-s imposes on $|\mcb O_{\mathbb F_m}(k,a)|$, also when condition $(\ref{star})$ fails:
	\begin{cor}\label{main coro}
		Using the notation of Lemma \ref{good lem}, define for each $0\leq i\leq j\leq u$, 
		\[
		\sigma_k(i,j):=\sum_{i< t\leq j}\mathrm{Max}\{(k-i)m+a+1,0\}-s_i
		\] 
		Let $i_0:=0$ and define inductively $i_t:= \mathrm{Min}\{j>i_{t-1}\ |\ \sigma_{k-i_{t-1}}(i_{t-1},j)<0\}$. Then, calling $Z$ the $0$-dimensional subscheme of $\mathbb F_m$ given by the union of the $S_i$ and $\mc I_Z$ its ideal sheaf,
		\begin{enumerate}[label=$\roman*)$]
			\item if $i_t>k+1$ for some $t$, then  $h^0(\mathbb F_m, \mc I_Z(k,a))=0$;
			\item otherwise, $h^0(\mathbb F_m, \mc I_Z(k,a))=h^0(\mathbb F_m, \mcb O_{\mathbb F_m}(k,a))-\sum_{i\leq u} s_i-\sum_{t}\sigma_{k-i_{t-1}}(i_{t-1},i_t)$.
		\end{enumerate}
	\end{cor}
	\begin{proof}
		Part $i)$ of the corollary immediately follows by noticing that, getting the matrix representing the evaluation morphism in a form similar to the one in $(\ref{block triang})$, the only element in the kernel is $0$.\\
		Similarly part $ii)$ is obtained by getting the matrix representing the evaluation morphism in a form similar to the one in $(\ref{block triang})$ and noticing that, whenever $i_t$ is defined, its rank drops by $\sigma_{k-i_{t-1}}(i_{t-1},i_t)$ as the matrix will be block diagonal with each block having rank 
		\[\sum_{i_{t-1}<\ell\leq i_t}(k-l)m+a+1\]
	\end{proof}
	\section{COMPUTING THE SCROLLAR INVARIANTS}
	Thanks to the lemmas proved in the previous sections, we may now proceede to proving the first two theorems:
	\begin{proof}[\bd{Proof of Theorem \ref{thm 1}}]
		Notice first that, since $\mc M_g(\mathbb F_m,(k,a))^{bir}$ is irreducible, proven that given a general point $S$ in $\mc Hilb^\delta(\mathbb F_m)$ there exists some $f:C\ra \mathbb F_m$ in $\mc M_g(\mathbb F_m,(k,a))^{bir}$ whose image is singular only in $S$,  the theorem follows by applying Lemma \ref{interpolating general points} and Remark \ref{good remark}, to compute:
		\[f(n):=h^0(\mc I_S(k-2,m+a-1-n)-h^0(\mc I_S(k-2,m+a-1-n)\]
		\vskip 1 em
		Take now $S$ as above. By assumption, $3\delta\leq \binom{k+1}{2}m+(k+1)(a+1)-1=|\mcb O_{\mathbb F_m}(k,a)|$. Thus, since being singular at a point imposes at most $3$ independent conditions on $|\mcb O_{\mathbb F_m}(k,a)|$,
		the sublinear system $\Sigma\subseteq |\mcb O_{\mathbb F_m}(k,a)|$ consisting of curves that are singular over $S$ is non empty.
		\vskip 1 em
		Now if there exists some $D\in \Sigma$ that does not contain the directrix of $\mathbb F_m$, calling $D_1,\dots,D_r$ its irreducible components, $K_{\mathbb F_m}\cdot D_j<0$ for each $j$. Thus, by Proposition $4.1$ of \cite{arbarello1981}, either the general element of $\Sigma$ is irreducible or there is a smooth and irreducible curve of genus $1$ $E$ such that $\Sigma\subseteq |2E|$. In the latter case, calling $(a,b)$ such that $[E]=a\zeta+b\xi$, one would need 
		\[1=\binom{a}{2}m+(a-1)(b-1)\]
		This in turn implies that $(m,k,a)\in\{(2,4,0),\ (1,6,0),\ (1,4,2),\ (0,4,4)\}$, which contradicts our assumptions. In particular, calling $\mc S$ the Severi variety parametrising irreducible nodal curves in $\mathbb F_m$ having class $(k,a)$ and geometric genus $g$, the map:
		\[\mathrm{Sing} \mc S\ra \mc Hilb^{\delta}(\mathbb F_m)\ |\ \Gamma\mapsto \mathrm{Sing}(\Gamma) \]
		is dominant.
		\vskip 1 em
		The only thing left to prove is thus that $\Sigma$ does not have the directrix $\Delta$ as a fixed component. Note first that, since $S$ is assumed to be general, we may assume that no point of $S$ lies in $\Delta$. Notice moreover that we may assume that $m\geq 2$ as, if $m\leq 1$, $\Delta\cdot K_{\mathbb F_m}<0$ and thus having the directrix as a component would pose no obstruction in applying the previous reasoning. 
		\vskip 1 em
		Assume by contradiction that $\Sigma$ had $\Delta$ as a fixed component and let $r$ be its multiplicity.\\
		Call $\Sigma'$ the sublinear system of $|\mcb O_{\mathbb F_m}(k-r,rm+a)|$ given by curves that are singular over $S$. By assumption there is an element $D\in \Sigma'$ not having the directrix as a component. Thus we may apply Proposition $4.1$ of \cite{arbarello1981} to find a $D'\in \Sigma'$ that is either an irreducible nodal curve singular only over $S$ or a smooth elliptic curve counted twice. Since $m\geq 2$, the latter case cannot happen hence $D'$ is an irreducible nodal curve singular only over $S$. Choosing $D''$ as $D'\cup \Delta$, we would then get an element of the sublinear system of $|\mcb O_{\mathbb F_m}(k-(r-1),a+(r-1)m)|$ given by curves that are singular over $S$. Since $D''$ is connected (as $a+rm>0$), has intersection with $S$ contained in only one component and:
		\[\mathrm{dim}|D'+\Delta|:=\mathrm{dim}|\mcb O_{\mathbb F_m}(k-r,a+rm)\otimes \mcb O_{\mathbb F_m}(1,-m)|=\]
		\[=\binom{k-r+2}{2}m+(k-(r-1)+1)(a+(r-1)m+1)-1=\]
		\[= \binom{k-r+1}{2}m+(k-r+1)(a+rm+1)+a+(r-1)m>|D'|=|D'|+|\Delta|\]
		one may repeat the proof of Step II in the proof of Proposition $4.1$ in \cite{arbarello1981} to find an element $D_0\in|\mcb O_{\mathbb F_m}(k-r+1,a+(r-1)m)|$ that is irreducible and singular over $S$. In particular, if we take $D_1$ to be the union of $D_0$ and $(r-1)$ times $\Delta$, $D_1$ lies in $\Sigma$ and has the directrix as a component with multiplicity $r-1$, giving us the desired contradiction.
	\end{proof}
	\begin{oss}
		Note that, since $\mathrm{dim}(\mc Hilb^\delta(\mathbb F_m))=2\delta$, the condition $3\delta\leq |\mcb O_{\mathbb F_m}(k,a)|$ is actually necessary to have dominance of the map $\mc S\ra \mc Hilb^\delta(\mathbb F_m)$. The above proof shows that it is also sufficient.
	\end{oss}
	\begin{ex}
		Taking for example $C$ to be the normalisation of a general curve $\Gamma$ of geometric genus $8$ and degree $6$ in $\Pb^2$, one has that its scrollar invariants with respect to a projection from a point not in $\Gamma$ are $(1,2,3,3,4)$ as one can see it as the normalisation of  a general curve of geometric genus $8$ and class $(6,0)$ in $F_1$.\\ 
		Note that these are indeed invariants of the cover and not of the curve. Indeed, projecting from a point on the curve would give $(2,3,3,4)$ while projecting from one of the nodes would yield $(3,4,4)$ as one may see them respectively as the normalisation of a general curve of geometric genus $8$ and class $(5,1)$  (resp. $(4,2)$) in $F_1$.
	\end{ex}
	As a corollary, since the Tschirnhausen bundle of a general $k:1$ cover of $\Proj$ is balanced (see for example \cite{ball89}) one has the following:
	\begin{cor}
		Let $f:C\ra \Proj$ be a general $k:1$ cover of $\Proj$. Then if $f$ factors as:
		\[
		\begin{tikzcd}
			C \arrow[r, "\nu"] \arrow{dr}[swap]{f} & \mathbb F_m \arrow[d,"\pi"]\\
			& \Proj
		\end{tikzcd}
		\]
		the image of $\nu$ has class $(k,a)$ with $a\geq \frac{g}{k-1}+1-m$.
	\end{cor}
	Note that, since by the proof of the previous theorem, a general point $f:C\ra \mathbb F_m$ of $\mc M_g(\mathbb F_m,(k,a))$ satisfying the hypothesis of Theorem \ref{thm 1} has as image an irreducible nodal curve with nodes in general position, we may use Lemma \ref{interpolating general points} to also compute the splitting types of $\mcb O_C(\Delta)$. We may then proceed to also prove Corollary \ref{cor 1}:
	\begin{proof}[\bd{Proof of Corollary \ref{cor 1}}] The only thing left to find is the splitting type of $\mcb O(\Delta)$.
		Twisting the exact sequence in $(\ref{h0 interp})$ by $\mcb O(-\Delta)$ and taking cohomology one gets:
		\[0\ra H^0(\mathbb F_m,\mc I_S(k-3,2m+a-2-n))\ra H^0(C,\omega_C\otimes \mc A^{\otimes -n}\otimes \nu^*\mcb O(-\Delta))\ra\]
		\[\ra H^1(\mathbb F_m,\mcb O_{\mathbb F_m}(-3,2m-2-n))\ra H^1(\mathbb F_m,\mc I_S(k-3,2m+a-2-n))\]
		Now applying Leray's spectral sequence one gets that ${H^1(\mcb O_{\mathbb F_m}(-3,2m-2-n))\cong H^0(\mcb O_{\Proj}(m-2-n))}$. If $n\leq m-2$, by our assumptions on $\delta$, we have that:
		\[h^0(\mathbb F_m,\mcb O_{\mathbb F_m}(k-3,2m+a-2-n))=\binom{k-1}{2}m+(a+m-1-n)(k-2)\geq \delta\]
		Thus, by Lemma \ref{interpolating general points}, $h^1(\mathbb F_m,\mc I_S(k-3,2m+a-2-n))=0$, implying that:
		\[h^0(C,\omega_C(-\Delta)\otimes \mc A^{\otimes -n})= m-1-n+\binom{k-1}{2}m+(a+m-1-n)(k-2)-\delta\]
		Setting $n=0$ and using Riemann-Roch gives $h^0(C,\mcb O_C(\Delta))=1$. Using the above equations, one sees that $m$ is the minimal $n$ such that:
		\[h^0(C,\omega_C(-\Delta)\otimes \mc A^{\otimes -n+1})-h^0(C,\omega_C(-\Delta)\otimes \mc A^{\otimes -n})<k-1\]
		Since the splitting type $(g_1,\dots,g_k)$ of $\mcb O_C(\Delta)$ satisfies $g_{i+1}\geq -e_{i}$ for each $i\leq k-1$ and $\sum_i g_{i+1}+e_i=a$, the previous theorem tells us that the $g_{i+1}=-e_i$ for each $i>1$.\\ 
		It is now easy to see that, with the same notation for $l,d,j$ as in Theorem \ref{thm 1}:
		\[h^1(C,\mc End(\nu_*\mcb O_C(\Delta)))=h^1(C,\mc End(\nu_*\mcb O_C))+(k-3)a\]
		\[h^1(C,\mc Hom(\nu_*\mcb O_C,\nu_*\mcb O_C(\Delta)))=h^1(C,\mc End(\nu_*\mcb O_C))-a\]
		\[h^1(C,\mc Hom(\nu_*\mcb O_C,\nu_*\mcb O_C(\Delta))(m))=h^1(C,\mc End(\nu_*\mcb O_C))+\binom{k}{2}m-\binom{l}{2}m-dl-j-(k-1)-a\]
		Using $g=\binom{k}{2}+(k-1)(a-1)-\delta=\binom{k}{2}+(k-1)(a-1)-\binom{l}{2}m-dl-j$ one then gets the statement of the corollary.
	\end{proof}
	
	The following proposition is the natural generalisation of Theorem $1$ in \cite{coppens2021_scrollar}. Although under some generality conditions this will be a special case of Theorem \ref{main thm}, we include it because it describes the scrollar invariants of every nodal curve having nodes arranged in such a fashion.
	\begin{prop}\label{prop cop}
		Let $C$ be the normalisation of an irreducible nodal curve $\Gamma\subseteq \mathbb F_m$ of class $(k,a)$ and let $S:=\mathrm{Sing}(\Gamma)$. Let for each $i=1,\dots k-1$, $C_i\subseteq \mathbb F_m$ be a section of $\pi:\mathbb F_m\ra \Proj$ different from the diurectrix of $\mathbb F_m$ and assume that $S=S_1\cup\dots\cup S_{k-1}$ where for each $i$, $S_i$ is a divisor of degree $s_i$ on $C_i$ with $s_{i}\leq \mathrm{Max}(0,s_{i-1}-m)$ and $s_1\leq (k-1)m+a+1$. Let $j_1:=\lfloor \frac{s_\ell}{m}\rfloor$ and $i_1=\#\{t\geq 0 \ |\ s_{1+t}=s_{1}-t m\}$.
		Then, the scrollar invariants of $C$ are:
		\[e_i=\begin{cases}
			im+a & \text{if } 1\leq i\leq j_1\\
			(k-1)m+a-s_1                                 & \text{if } j_1<i< j_1+i_1
		\end{cases}\]
		The other ones are the last $k-j_1-i_1$ scrollar invariants of a curve of class $(k-i_1,a)$ having singular locus contained in $S':=S_{i_1+1}\cup \dots \cup S_{k-1}$
	\end{prop}
	\begin{proof}
		Note first that, since the case $m=0$ is proven in \cite{coppens2021_scrollar}, we only need to prove the case $m>0$. By assumption, if $s_{1}\leq (k-2)m+a+m-1-n$, for each $i$,
		${s_{i}\leq (k-i)m+a-1-n}$.\\ 
		Thus, for any such $n$, by Lemma \ref{easy h0}, one has that:
		\[h^0(\mathbb F_m,\mc I_S(k-2,m+a-2-n))=h^0(\mathbb F_m,\mcb O_{\mathbb F_m}(k-2,a+m-2-n))-\delta\]
		For each $i\leq (k-1-\frac{s_1+1}{m})$, the scrollar invariants of $C$ are hence those of a smooth curve
		in $\mathbb F_m$ of class $(k,a)$.
		\vskip 1 em
		If $n\geq (k-1)m+a-s_{1}$, the intersection pairing on $A^\bullet(\mathbb F_m)$, tells us
		that any global section of $\mc I_S(k-2,a+m-2-n)$ vanishes on $C_{1},\dots,C_{i_1}$. Indeed, calling $\sigma$ any such section, calling $\ell_0$ the first index such that $C_{\ell_0}\nsubseteq V(\sigma)$, one would have 
		\[C_{\ell_0}.V(\sigma):=\zeta\cdot ((k-1-\ell_0)\zeta+(m+a-n-2)\xi)<s_{1}-(\ell_0-1)m:=s_{\ell_0}\] 
		Thus global sections of $\mc I_S(k-2,a+m-2-n)$ are determined by global sections of\\ 
		${\mc I_{S'}(k-2-i_1,a+m-2-n)}$, where $S'$ is given by the union of all points of $S$ not in $C_{1},\dots,C_{i_1}$. In particular, for any $n> (k-1)m+a-s_{1}$, $f(n)$ is the same as that of the normalisation of a curve of class $(k-i_1,a)$ having singular locus satisfying the hypothesis of the lemma. 
		\vskip 1 em
		From this it follows that, determined $f((k-1)m+a-s_{1})$, one may inductively compute the scrollar invariants of $C$. By the above remarks one has:
		\[f((k-1)m+a-s_{1})=h^0(\mathbb F_m,\mcb O_{\mathbb F_m}(k-2,s_{1}-(k-2)m-1))-\delta+\]
		\[-(h^0(\mathbb F_m,\mcb O_{\mathbb F_m}(k-2-i_1,s_{1}-(k-2)m-2))-(\delta-\sum_{1\leq i\leq i_1} s_{i}))=\]
		\[h^0(\mathbb F_m,\mcb O_{\mathbb F_m}(k-2,s_{1}-(k-2)m-1))-h^0(\mathbb F_m,\mcb O_{\mathbb F_m}(k-2-i_1,s_{1}-(k-2)m-2))-\sum_{1\leq i\leq i_1} s_{i}\]
		which is either $j_{1}+1-i_{1}$ if $s_{1}\neq j_{1}m$ or $j_{1}+2-i_1$ if $s_{1}=j_{1}m$. Notice that the discrepancy here
		is just due to the fact that, in the first case, $(k-j_{1}-1)m+a$ (which, by induction, is the $(k-j_{1}-1)$-th invariant of a curve on $\mathbb F_m$ of class $(k,a)$ staisfying the hypothesis of the lemma) equals $(k-1)m+a-s_{1}$.
	\end{proof}
	\begin{ex}
		Choosing for example $k=7$, $s_1=3m+\alpha_1$, $s_2=m+\alpha_2$, with $\alpha_1,\alpha_2<m$ we have that the scrollar invariants of a curve having singular locus contained in the union of $2$ sections $C_1,C_2$ of $\pi:\mathbb F_m\ra \Proj$ with $s_1$ points in $C_1$ and $s_2$ points in $C_2$ are:
		\[(a+m,a+2m,a+3m-\alpha_1,a+3m, a+4m-\alpha_2, a+4m)\]
	\end{ex}
	\vskip 1 em
	Notice now though that, since lying on a curve of class $\zeta$ defines a divisor in $\mc Hilb^\delta(\mathbb F_m)$ for any $\delta>m+1$, if $m>0$ this proposition cannot describe the invariants of the normalisation of a general curve in $\mathbb F_m$. This is much weaker than what happens for $\Proj\times \Proj$ in \cite{coppens2021_scrollar} as, since that only dealt with $m=0$, there one may take $\delta\leq (k-1)$. The second Theorem of this paper aims to fill in the gap between the situation described above and the one where the curve is general in the Severi variety:
	\begin{proof}[\bd{Proof of Theorem \ref{main thm}}]
		Again the idea will be to compute $h^0(\mc I_S(k-2,m+a-2-n))$ for arbitrary $n$, which can be easily done using Lemma \ref{good lem} and Corollary \ref{main coro}. Notice also that we need only to compute it for $n\leq n_1$ as, proceeding as in the proof of Corollary \ref{main coro}, one sees that, if 
		\[\sum_{i\leq \ell}s_i>\sum_{i\leq \ell}\mathrm{Max}\{(k-i)m+a-1-n,0\}\]
		the matrix in $(\ref{block triang})$ does not have full rank and, choosing $i_1$ to be the first integer for which this happens, one sees that any section that vanishes over such an $S$ is determined by a section of $\mcb O_{\mathbb F_m}(k-2-i_1,m+a-2-n)$ vanishing over $S'$, where $S'$ is the subset of $S$ given by those points that lie in $C_i$ for $i<k-1-i_1$. In particular, in this case, its rank and consequently $h^0(\mathbb F_m,\mc I_S(k-2,m+a-2-n))$, may be determined inductively.
		\vskip 1 em
		As noted in the previous proposition, whenever passing through $S$ imposes independent conditions on $|\mcb O_{\mathbb F_m}(k,a)|$, $f^S_{k,a}=f^{\emptyset}_{k,a}$. Thus, the only thing left to prove the Theorem is to compute $f^S_{k,a}(n_1)$. By Lemma \ref{good lem} and Corollary \ref{main coro}, this is equal to:
		\[
		h^0({\mathbb F_m},\mc I_S(k-2,a+m-1-n_1))-h^0(\mathbb F_m,\mc I_S(k-2,a+m-2-n_1))=\]
		\[\left(h^0(\mcb O_{\mathbb F_m}(k-2,a+m-1-n_1))-h^0(\mcb O_{\mathbb F_m}(k-2,a+m-2-n_1))\right)+\sum_t\sigma_{k-i_{t-1}}(i_{t-1},i_t,n_1)\]
		Now, since for each $i=1,\dots, k-1$, the $i$-th scrollar invariant of a smooth curve of class $(k,a)$ is $im+a$, choosing $r_1$ so that $r_1m+a\leq n_1<(r_1+1)m+a$, \[
		h^0({\mathbb F_m},\mcb O_{\mathbb F_m}(k-2,a+m-1-n_1))-h^0(\mathbb F_m,\mcb O_{\mathbb F_m}(k-2,a+m-2-n_1)=k-r_1
		\]
		which in turn shows that the scrollar invariants are those described in the statement of the Theorem.
	\end{proof}
	\begin{oss}
		Note that, if we take $\delta\leq \mathrm{Min}(p_a(k,a), \frac{1}{3}h^0(\mcb O_{\mathbb F_m}(k,a)))$, choosing each of the $s_i\leq m+1$ we get back the scrollar invariants of the normalisation of a general curve of class $(k,a)$ and geometric genus $p_a(k,a)-\delta$. This is consistent with what we showed in Theorem \ref{thm 1} as, with those assumptions on the genus, the nodes of such a curve are in general position.
		\vskip 1 em
		In particular, Corollary \ref{main coro} bridges the gap between Theorem \ref{thm 1} and Proposition \ref{prop cop}.
	\end{oss}
	\begin{ex}
		As an example, take $(k,a)=(9,4)$, $(s_1,\dots,s_{u})=(18,18,10,6,2)$.
		Following the procedure illustrated in the proof, one gets $n_1=16$, $r_1=3$, $j_1=i_1=2$, $\delta_1=2$. Hence the first 5 scrollar invariants are $8,12,16,16,16$. The other ones may be computed by considering those of a curve having $(k_2,a_2)=(7,4)$ and $(s_1',\dots,s'_u)=(10,6,2)$. Repeating the reasoning of the proof in this case yields $n_2=18$, $r_2=0$, $j_2=3$, $\delta_2=3$. Thus the scrollar invariants of such a curve are:
		\[(8,12,16,16,16,18,18,18)\]
		Note that the normalisation of a general curve of this class and geometric genus would instead have invariants:
		\[(8,12,16,17,17,17,17,18)\]
	\end{ex}
	\vskip 3 em
	\section{CURVES IN $\mathbb F_m$ WITH NODES IN PRESCRIBED POSITION}
	We end off this note by showing the existence of curves in $\mathbb F_m$ having nodes positioned as prescribed in Theorem \ref{main thm}. Restricting to the case of $m=0$ this greatly improves the bounds on the genus given in Proposition 1 of \cite{coppens2021_scrollar} and in the Third Claim of \cite{bal02}. In what follows we shall assume that $k\geq 5$. If $k\leq 4$, one may mimic the proof of Theorem $2$ in \cite{coppens2021_scrollar} to get the bound on $g$ required in the statement of Theorem \ref{final thm}.
	\vskip 2 em
	\begin{lem}\label{main lem2}
		Let $u\geq 0$ and let $s_1\geq\dots\geq s_u$ be a sequence of non increasing integers. Then, for any $\ell\geq 1$, $k\geq \ell-1$, $a\geq \big\lceil \frac{u}{\ell}\big\rceil s_1+s_{u}$ and general sections of $\pi:\mathbb F_m\ra \Proj$, $C_1,\dots, C_{u}$, choosing a subset $S_i\subseteq C_i$ of $s_i$ general points and calling $S:=S_1\cup\dots\cup S_u$, the linear system $|\mc I_S(k,a))|$ does not have a fixed component.
	\end{lem}
	\begin{proof}
		We shall prove first that, for a general choice of sections of $\pi$ and general points on these sections, $\Sigma:=|\mc I_S(u,a))|$ does not have any fixed component of class $(1,0)$ or $(0,1)$.
		\vskip 1 em
		If there were a component of class $(1,0)$ it would need to be one of the sections $C_i$. Now, this implies that, choosing any $2$ points in $C_i$ different from those in $S_i$, these do not impose any conditions on $\Sigma:=|\mc I_S(k,a)|$. In particular, we need only to prove that, choosing $2$ general points in each of the sections, these impose $2u$ conditions on $\Sigma$. Again doing this only requires exhibiting one example. 
		\vskip 0.5 em
		Choose $\lambda_1,\dots,\lambda_u\in \mcb k$ and $h_i=\lambda_it_0^m$. Choose coordinates $\left(([t_{0,i,j}:t_{1,i,j}])_{i=1}^{s_j+2}\right)_{j=1}^u$ on $\left(\Proj\right)^{s+2u}$ and define:
		\[ev_S:\prod_{i=0}^u H^0(\mcb O_{\Proj}((k-i)m+a))\ra \mcb k^{\oplus s+2u}\]
		\[(f_k,\dots,f_0)\mapsto \left(\left(\sum_{i=0}^k f_i(t_{0,\ell,j},t_{1,\ell,j})h_{\ell}(t_{0,\ell,j},t_{1,\ell,j})^i\right)_{j=1}^{s_i+2}\right)_{l=1}^u\]
		By construction, fixing $\mbd q\in \left(\Proj\right)^{s+2u}$ this map corresponds to evaluating sections of $\mcb O_{\mathbb F_m}(k,a)$ at $S(\mbd q):=\bigcup_{i=1}^u S(\mbd q)_i$ where $S(\mbd q)_i$ is the set of points of $C_i:=V(x_1-h_ix_0)$ that project to $\{q_{j,i}\}_{j=1}^{s_i+2}$. Thus, we need only to show that for some $\lambda_1,\dots,\lambda_u$, $[t_{0,i,j}:t_{1,i,j}]$ $ev_S$ has maximal rank. Choosing $\lambda_\ell=\lambda_j$ for $\ell=j$ $mod$ $\ell+1$, Lemma \ref{easy h0} readily gives the result.
		\vskip 1 em
		Similarly if $\Sigma$ had a fixed component of class $(0,1)$, it would be a line passing through one of the points of $S$. Now, choosing the $\lambda$-s and $h_i$-s as above, let $\lambda\in \mcb k$ be distinct from the $\lambda_i$-s, $h:=\lambda t_0^m$ and consider for each $(b,c)\in \{(i,j)\ |\ 1\leq i \leq s_j,\ 1\leq j\leq u\}$:
		\[ev_S^{b,c}:\prod_{i=0}^k H^0(\mcb O_{\Proj}((k-i)m+a))\ra \mcb k^{\oplus s+1}\]
		\[(f_k,\dots,f_0)\mapsto \left(\left(\left(\sum_{i=0}^k f_i(t_{0,\ell,j},t_{1,\ell,j})h_{l}(t_{0,\ell,j},t_{1,\ell,j})^i\right)_{j=1}^{s_i}\right)_{\ell=1}^u,\sum_{i=0}^k f_i(t_{0,b,c},t_{1,b,c})h(t_{0,b,c},t_{1,b,c})^i\right)\]
		As above, $ev_S^{b,c}$ corresponds to the evaluation of global sections of $\mcb O_{\mathbb F_m}(k,a)$ at $S$ and at a point lying on the same line as the $b$-th point on $C_c$. Hence, reasoning as above, for general choices of the $\lambda$-s and the points of $S$, $ev_S^{b,c}$ has maximal rank. The finiteness of the set of $(b,c)$-s then implies that, for a general choice of $S$ and of the sections, $\Sigma$ will not have a fixed component of type $(0,1)$.
		\vskip 2 em
		Now, if $\Sigma$ had a fixed component, since it cannot be of type $(0,1)$ or $(1,0)$, it will intersect the general member of the pencil $\{C_\lambda:=V(x_1-\lambda t_0^m x_0)\}_{\lambda\in \mcb k}$ transversely in one or more points.
		\vskip 0.5 em
		\textit{Note:} We have to rule out the case of the component being of type $(1,0)$ to be able to treat also the case $m=0$. Notice also that the fixed component cannot be the directrix of $\mathbb F_m$ as, by assumption, no points of $S$ lie on it.
		\vskip 0.5 em
		Reasoning as above, this would imply that for general $\lambda\in \mathbb A^1_{\mcb k}$, there exists $t_0,t_1$ such that the following matrix does not have maximal rank:
		\[\left(\begin{array}{cccccc}
			t_{0,1,1}^{um+a} & \dots & t_{1,1,1}^{um+a} & \lambda_1t_{0,1,1}^{(u-1)m+a} & \dots & \lambda_1^ut_{1,1,1}^a\\
			\dots & \dots & \dots & \dots & \dots & \dots\\
			t_{0,s_1,1}^{um+a} & \dots & t_{1,s_1,1}^{um+a} & \lambda_1t_{0,s_1,1}^{(u-1)m+a} &\dots & \lambda_1^ut_{1,s_1,1}^a\\
			t_{0,1,2}^{um+a} & \dots & t_{1,1,2}^{um+a} & \lambda_2t_{0,1,2}^{(u-1)m+a} &\dots & \lambda_2^ut_{1,1,2}^a\\
			\dots &  \dots & \dots & \dots & \dots & \dots\\
			t_{0,s_k,k}^{um+a} & \dots & t_{1,s_k,k}^{um+a} & \lambda_kt_{0,s_k,k}^{(u-1)m+a} &\dots & \lambda_k^ut_{1,s_k,k}^a\\
			t_{0}^{um+a} & \dots & t_{1}^{um+a} & \lambda t_{0}^{(u-1)m+a} &\dots & \lambda^ut_{1}^a
		\end{array}\right)\]
		In particular, for any $\lambda$, the $(s+1)\times (s+1)$ minors of the above matrix all have a common root. Hence, for example by Theorem $8$ of \cite{subresultant}, their $0$-th subresultant is identically $0$ as a polynomial in $\lambda$.
		At this point, specialising to the case $\lambda_i=\lambda_j$ whenever $i=j$ $mod$ $\ell$ and reasoning as before gives the desired contradiction.
	\end{proof}
	\begin{oss}
		The proof of the previous lemma also shows that $h^1(\mc I_{S\cup p}(k,a))=0$ for a general $p\in \mathbb F_m$.
	\end{oss}
	As a corollary we get:
	\begin{cor}\label{step 2}
		With the notation of the previous lemma, for general $C_1,\dots,C_{u}$, $\ell\geq 1$, $k\geq \ell-1$ and $a\geq \left(\big\lceil \frac{u}{\ell}\big\rceil +1\right)s_1$, the linear system $\Sigma:=|\mc I_S(k,a)|$ has no base point outside of $S$ and its general member is irreducible.
	\end{cor}
	\begin{proof}
		Notice that, by the previous lemma, the linear system has no fixed component. Now, since the linear system is not composed with a pencil (for dimension reasons one can find an element of the linear system with a line as a component and thus find a member of $\Sigma$ whose components have non-proportional numerical class), by Theorem 5.3 of \cite{kleimBertini}, the general member of $\Sigma$ is irreducible. In particular, since $\mcb k$ is infinite and hence $\mcb k(t_1,\dots,t_{\mathrm{dim}\Sigma})$ Hilbertian and since, calling $X_\eta$ the generic fibre of the universal curve over $\Sigma$, for $c\gg 0$, $\mathrm{deg}(\mc I_p(1,c)_{|_{X_\eta}})$ is greater than $\#S+\mathrm{deg}(\omega_{X_\eta})$,
		the fibre over the generic point of $\Sigma$ of:
		\[p_1:\Phi_C:=\{(f,q)\in \Sigma\times C\ |\ f(q)=0\}\ra \Sigma\]
		is irreducible (for example by Lemma $8.6.13$ in \cite{kollar}). Hence, for a general $C\in |\mc I_p(1,c)|$, so is $\Phi_C$.
		\vskip 2 em 
		Choosing any such smooth $C$, the second projection $p_2:\Phi_C\ra C$ is then a dominant morphism from a variety to a nonsingular curve and is hence flat. In particular, the dimension of its fibres is constant. Thus, since by the previous lemma the fibre over a general point $q\in C$ has the expected dimension, so do all the others.  
	\end{proof}
	Notice also that, since Lemma \ref{easy h0} can be used not only to show that passing thorugh some points in the $C_i$-s imposes independent conditions on $|\mcb O_{\mathbb F_m}(l,a)|$ but also that the same is true for having the same jets at those points as the $C_i$-s do, one can analogously prove:
	\begin{cor}\label{smooth coro}
		With the notation of the previous corollary, the general member of the linear system $|\mc I_S(k,a)|$ is smooth.
	\end{cor}
	\begin{oss}
		The above corollaries do not give sharp bounds on $a$ as one can go through the proof of Lemma \ref{main lem2} and easily get a better bound.
	\end{oss}
	The above corollary then allows us to improve Step I and II in the proof of Lemma $1$ of \cite{bal02}.
	\begin{lem}\label{step 1}
		Let $u$ be a non-negative integer and let $C_1,\dots, C_{u}$ be general sections of $\mathbb F_m\ra \Proj$. Let for each $i$, $S_i$ be a collection of $s_i$ general points in $C_i$ with $s_1\geq\dots \geq s_u$. Let $S^{(1)}$ be the first order thickening of $S:=S_1\cup\dots \cup S_u$ in $\mathbb F_m$.
		Then, for each $a\geq 2(\big\lceil\frac{2u}{k}\big\rceil+1)s_1$, $\ell\geq k$
		\[h^1(\mathbb F_m,\mc I_{S^{(1)}}(\ell,a))=0\]
	\end{lem}
	\begin{proof}
		Let $C_0$ be a general curve of class $(\big\lceil \frac{k}{2}\big\rceil-1,(\big\lceil\frac{2u}{k}\big\rceil+1)s_1)$ passing through $S$. By the previuous corollary, the residual scheme of $S^{(1)}$ with respect to $C_0$ is $S$. 
		This translates to the exactness of:
		\[0\ra \mc I_S\left(\ell-\bigg\lceil \frac{k}{2}\bigg\rceil+1,a-\left(\bigg\lceil\frac{2u}{k}\bigg\rceil+1\right)s_1\right)\ra \mc I_{S^{(1)}}(\ell,a)\ra \mc I_{2S,C_0}\left(\ell,a\right)\ra 0\]
		Now, $2S$ is a Cartier divisor of degree $2\delta$ on $C_0$.
		Since, by definition, one has that:
		\[\mathrm{deg}(\mc I_{2S,C_0}(\ell,a))=C_0.(\ell\zeta+a\xi)-2\sum_{i\leq u}s_i=\]
        \[=\left(\bigg\lceil \frac{k}{2}\bigg\rceil-1\right)\ell m+\left(\bigg\lceil\frac{2u}{k}\bigg\rceil+1\right)\ell s_1+\left(\bigg\lceil \frac{k}{2}\bigg\rceil-1\right)a-2\sum_{i\leq u}s_i\geq\]
		\[\geq \left(\bigg\lceil \frac{k}{2}\bigg\rceil-1\right)km+\left(\bigg\lceil\frac{2u}{k}\bigg\rceil+1\right)ks_1+2\left(\bigg\lceil\frac{2u}{k}\bigg\rceil+1\right)s_1\left(\bigg\lceil \frac{k}{2}\bigg\rceil-1\right)-2us_1\geq\]
		\[\geq 2\binom{\lceil \frac{k}{2}\rceil-1}{2}m+2\left(\bigg\lceil \frac{k}{2}\bigg\rceil-1\right)\left(\left(\bigg\lceil\frac{2u}{k}\bigg\rceil+1\right)s_1-1\right)\] and, by adjunction, the degree of the canonical divisor of $C_0$ is given by
		\[2\binom{\lceil \frac{k}{2}\rceil-1}{2}m+2\left(\left(\bigg\lceil\frac{2u}{k}\bigg\rceil+1\right)s_1-1\right)\left(\bigg\lceil \frac{k}{2}\bigg\rceil-2\right)-2,\]
		the line bundle $\mc I_{2S,C_0}(\ell,a)$ is non-special. In particular, $h^1(\mc I_{2S,C_0}(l,a))=0$.\\
		The lemma now follows by noticing that, since $\ell\geq k$, $a\geq2(\lceil\frac{2u}{k}\rceil+1)s_1$, we also have the vanishing of $h^1(\mc I_S(\ell-\lceil \frac{k}{2}\rceil+1,a-(\lceil\frac{2u}{k}\rceil+1)s_1))$.\\
	\end{proof}
	An immediate consequence of the previous lemma is the following corollary. The proof follows the same path as the one of Step II of \cite{bal02}. We decided to include it for completeness:
	\begin{cor}\label{2nd order}
		With the notation of the previous lemma, let $p$ be a point of $S$ and let $p^{(2)}$ be the second order thickening of $p$ in $\mathbb F_m$. Then, for each $a\geq 2(\lceil\frac{2u}{k}\rceil+1)s_1+1$, $\ell\geq k$:
		\[h^1(\mc I_{S^{(1)}\cup p^{(2)}}(\ell,a))=0\]
	\end{cor}
	\begin{proof}
		Let $L$ be the fibre of $\pi:\mathbb F_m\ra \Proj$ passing though $p$. By construction, the residual scheme of $S^{(1)}\cup p^{(2)}$ with resepct to $L$ is $S^{(1)}$. Hence we have an exact sequence:
		\[0\ra \mc I_{S^{(1)}}(\ell,a-1)\ra \mc I_{S^{(1)}\cup p^{(2)}}(\ell,a)\ra \mc I_{3p,L}(\ell,a)\ra 0\]
		By the above lemma $h^1(\mc I_{S^{(1)}}(\ell,a-1))=0$. Since $k\geq 3$, we also have $h^1(\mc I_{3p,L}(\ell,a))=0$.
	\end{proof}
	\vskip 2 em 
	Thanks to the above lemmas, we may now proceed with the proof of Theorem \ref{final thm}:
	\begin{proof}[\bd{Proof of Theorem \ref{final thm}}]
		Let $S^{(1)}$ be the first order thickening of $S$ in $\mathbb F_m$. By definition, ${C\in|\mc I_{S^{(1)}}(k,a)|:=\Sigma}$ if and only if it is singular at $S$. Define also:
		\[\Sigma_1:=\bigg|\mc I_{S}\left(\bigg\lceil\frac{k}{2}\bigg\rceil -1,\left(\bigg\lceil\frac{2u}{k}\bigg\rceil+1\right)s_1\right)\bigg| \qquad \Sigma_2:=\bigg|\mc I_{S}\left(\bigg\lfloor\frac{k}{2}\bigg\rfloor +1,\left(\bigg\lceil\frac{2u}{k}\bigg\rceil+1\right)s_1\right)\bigg|\]
		Choosing $L_1,\dots,L_{a-2\left(\lceil\frac{2u}{k}\rceil+1\right)s_1}$ fibres of $\pi:\mathbb F_m\ra \Proj$,  by construction one has that 
		\[\Sigma_1+ \Sigma_2+\sum_{i\leq a-2(\lceil\frac{2u}{k}\rceil+1)s_1} L_i\leq \Sigma\]
		Corollaries \ref{step 2} and \ref{smooth coro} guarantee that a general member of $\Sigma$ is singular only over $S$.
		Corollary \ref{2nd order} now shows that a general member of $\Sigma$ has only nodes as singularities.
		\vskip 2 em
		Finally, since both $\Sigma_1$ and $\Sigma_2$ are positive dimensional, $\Sigma$ does not have a fixed component. By Theorem 5.3 of \cite{kleimBertini}, if a general member of $\Sigma$ were to be reducible, $\Sigma$ would be composed with a pencil. Since, for general $C_1\in \Sigma_1$, $C_2\in \Sigma_2$ the components of $C_1\cup C_2\cup L_1\cup\dots\cup L_{a-2(\lceil\frac{2u}{k}\rceil+1)s_1}$ have non-proportional numerical class, this cannot be the case. Hence the general member of $\Sigma$ is irreducible, nodal and singular only over $S$.
	\end{proof}
	From this, working exactly as in the proof of the main theorem of \cite{bal02} or (equiavelently) as in the proof of Theorem $2$ of \cite{coppens2021_scrollar} one gets the proof of Corollary $\ref{coro P^1}$.
	\begin{oss}
		Notice that, although in Corollary \ref{coro P^1} we have managed to get a genus that depends linearly on k, this dependence still limits the normalised scrollar invariants we can obtain just using normalisations of nodal curves in $\Proj\times \Proj$. Indeed, calling $\ol e_i:=\frac{e_i}{g+k-1}$, for any curve as above, \[\ol e_{k-1}-\ol e_1=\frac{d}{g+k-1}< \frac{d}{(6d+1)(k-1)}\]
	\end{oss}
	\printbibliography
    \vskip 2 em
    \address{Humboldt Universität zu Berlin, Unter den Linden 6, 10117 Berlin}
    \email{riccardo.redigolo@hu-berlin.de}
\end{document}